\documentclass[a4paper,11pt]{article}

\usepackage{amsmath}
\usepackage{amsthm}
\usepackage{amsfonts}
\usepackage{amssymb}
\usepackage{t-angles}

\usepackage[all]{xy}
\usepackage{amscd}

\newcommand{\al}{\alpha}
\newcommand{\si}{\sigma}

\newcommand{\id}{\mathrm{id}}

\newcommand{\ot}{\otimes}

\newcommand{\trl}{\triangleleft}
\newcommand{\trr}{\triangleright}

\def\ppr{\rightharpoonup}
\def\ppl{\leftharpoonup}

\newcommand{\li}{{}_{1}}
\newcommand{\lii}{{}_{2}}

\newcommand{\lmo}{{}_{(0)}} 

\newcommand{\loo}{{}_{(0)}}
\newcommand{\loi}{{}_{(-1)}}

\newcommand{\lmoo}{{}_{(0)}}

\newcommand{\lmi}{{}_{(1)}}

\newcommand{\lmoi}{{}_{(-1)}}

\newcommand{\mo}{{}_{(0)}}
\newcommand{\mi}{{}_{(1)}}

\newcommand{\moi}{{}_{(-1)}}

\newcommand{\boo}{{}_{[0]}}
\newcommand{\bi}{{}_{[1]}}

\newcommand{\boi}{{}_{[-1]}}

\newcommand{\ppi}{{}_{<1>}}
\newcommand{\pii}{{}_{<2>}}
\newcommand{\qi}{{}_{\{1\}}}
\newcommand{\qii}{{}_{\{2\}}}

\newcommand{\alaa}{\alpha_{A}}
\newcommand{\beaa}{\beta_{A}}

\newcommand{\alvv}{\alpha_{V}}
\newcommand{\bevv}{\beta_{V}}
\newcommand{\alhh}{\alpha_{H}}
\newcommand{\behh}{\beta_{H}}

\def\rbiprod{{\cdot\kern-.33em\triangleright\!\!\!<}}
\def\lbiprod{{>\!\!\!\triangleleft\kern-.33em\cdot\, }}

\def\lrbiprod{{\ \cdot\kern-.60em\triangleright\kern-.33em\triangleleft\kern-.33em\cdot\, }}

\def\lprod{{>\!\!\!\triangleleft\kern-.33em\ \, }}

\newcommand{\lrcoprod}{{\,\blacktriangleright\!\!\blacktriangleleft\, }}

\allowdisplaybreaks
\newtheorem{theorem}{Theorem}[section]
\newtheorem{lemma}[theorem]{Lemma}
\newtheorem{proposition}[theorem]{Proposition}

\theoremstyle{definition}
\newtheorem{definition}[theorem]{Definition}

\title{Extending structures for BiHom-Frobenius algebras}
\author{Tao Zhang, Hui-Jun Yao}
\date{}

\begin{document}
 \maketitle

 \setcounter{section}{0}

\begin{abstract}
We introduce the concept of braided BiHom-Frobenius algebras and give the  cocycle bicrossproduct construction for BiHom-Frobenius algebras.
We find that the extending problem for BiHom-Frobenius algebras can be classified by non-abelian cohomology theory.

\end{abstract}

\tableofcontents

\section{Introduction}

We define concept of BiHom-Frobenius algebras as a BiHom-analogue of Frobenius algebras.
Roughly speaking, it is a vector space $A$ equipped simultaneously with a BiHom-associative algebra  $(A, \cdot,\al,\beta)$  and a BiHom-coassociative  coalgebra $(A, \Delta, \al,\beta)$  such that the following compatibility condition is satisfied,
\begin{equation}
\Delta(a b)=\alpha(a) b_{1} \otimes \beta\left(b_{2}\right)=\alpha\left(a_{1}\right) \otimes a_{2} \beta(b).
\end{equation}
Furthermore, the construction of the Drinfeld double for BiHom-Frobenius algebras was given  in \cite{HHS}.
The relationship among infinitesimal BiHom-bialgebras,  BiHom-pre-Lie algebra and the associative BiHom-Yang-Baxter equation was studied in \cite{Liu19,Liu20}.

On the other direction, the theory of extending structures for many types of algebras such as Lie algebra, associative algebra, Leibniz algebra,  Poisson algebra were well developed by A. L. Agore and G. Militaru in \cite{AM1,AM2,AM3,AM4,AM5,AM6}.
Recently, extending structures for 3-Lie algebras, Lie bialgebras,  infinitesimal bialgebras and Lie conformal superalgebras were studied  in \cite{Z2,Z3,Z4,ZCY}.
Until now, extending structures for Hom-algebra and Hom-bialgebra structures were not developed as well as the above-mentioned algebras.

In this paper, we study  extending structures  for BiHom-Frobenius algebra.
Let $A$ be a BiHom-Frobenius algebra and $E$ a BiHom-vector space containing $A$ as a subspace.
The extending problem is to describe and classify all BiHom-Frobenius algebra structures on $E$
such that $A$ is a sub(co)algebra or quotient (co)algebra of $E$.
To solve this problem, we introduce the concept of braided BiHom-Frobenius algebras.
This new concept can be seen as a BiHom-analogue of braided Hopf algebra in quantum group theory  \cite{Ra85,Ma95,BD99,BD01}
and braided Lie bialgebra in \cite{Ma00,So96}.
Using this new concept, we give a construction of the unified product theory to  solve the extending problem for BiHom-Frobenius algebras.
It is proved that this problem can be classified by using non-abelian cohomology theory.

This paper is organized as follows. In Section 2, we recall some definitions and fix some notations.
In Section 3, we introduce the concept of BiHom-Frobenius-Hopf bimodule and braided BiHom-Frobenius algebras.
Then we construct ordinary BiHom-Frobenius algebras from braided BiHom-Frobenius algebras via biproduct.
In  Section 4, we define the notion of matched pairs of  braided BiHom-Frobenius algebras and construct cocycle bicrossproduct BiHom-Frobenius algebras through two generalized braided  BiHom-Frobenius algebras.
In Section 5, we study the extending problems for BiHom-Frobenius algebras and prove that this problem can be classified by some non-abelian cohomology theory.

Throughout this paper, all vector spaces will be over a fixed field of character zero. The identity map of a vector space $V$ is denoted by $\id_V: V\to V$ or simply $\id: V\to V$.  A BiHom-vector space $(V,\alvv,\bevv)$ is a vector space together with two linear maps $\alvv,\bevv: V\to V$.

\section{Preliminaries}

\begin{definition}[\cite{GMMP}]
A BiHom-associative algebra is a 4-tuple $(A, \cdot, \alpha, \beta)$, where $A$ is a linear space and $\alpha, \beta: A \rightarrow A$ and $\mu: A \otimes A \rightarrow A$ are linear maps such that
\begin{equation}\label{def:ass-01}
\alpha \circ \beta=\beta \circ \alpha,\quad\alpha(x \cdot y)=\alpha(x) \cdot \alpha(y),\quad \beta(x \cdot y)=\beta(x) \cdot \beta(y)
\end{equation}
 and
\begin{equation}\label{def:ass-02}
\alpha(x) \cdot(y \cdot z)=(x \cdot y) \cdot \beta(z)
\end{equation}
for all $x, y, z \in A$ where $x \cdot y=\mu(x\ot y)$. The condition  \eqref{def:ass-02}  is called the BiHom-associativity condition and the maps $\alpha$ and $\beta$ are called the structure maps of $A$. In the follow text, we will write $x \cdot y$ as $x y$ and call a BiHom-associative algebra by BiHom-algebra for simplicity.
\end{definition}

A morphism $f:\left(A, \cdot_{A}, \alpha_{A}, \beta_{A}\right) \rightarrow\left(B, \cdot_{B}, \alpha_{B}, \beta_{B}\right)$ of BiHom-associative algebras is a linear map $f: A \rightarrow B$ such that $\alpha_{B} \circ f=f \circ \alpha_{A}, \beta_{B} \circ f=f \circ \beta_{A}$ and $f (x\cdot_{A} y)=f(x) \cdot_B f(y)$.


\begin{definition}[\cite{Liu20}]
A BiHom-coassociative coalgebra is a 4-tuple $(C, \Delta, \nu, \omega)$, in which $C$ is a linear space, $\nu, \omega: C \rightarrow C$ and $\Delta: C \rightarrow C \otimes C$ are linear maps, such that
\begin{equation}\label{def:coass-01}
\nu \circ \omega=\omega \circ \nu,\quad (\nu \otimes \nu) \circ \Delta=\Delta \circ \nu,\quad(\omega \otimes \omega) \circ \Delta=\Delta \circ \omega
\end{equation}
 and
\begin{equation}\label{def:coass-02}
(\Delta \otimes \nu) \circ \Delta=(\omega \otimes \Delta) \circ \Delta.
\end{equation}
The  above condition \eqref{def:coass-02} is called the BiHom-coassociativity condition and the maps $\nu$ and $\omega$  are called the structure maps of $C$.
A BiHom-coassociative coalgebra is simply called a BiHom-coalgebra in the following text.
\end{definition}

A morphism $f:\left(C, \Delta_{C}, \nu_{C}, \omega_{C}\right) \rightarrow\left(D, \Delta_{D}, \nu_{D}, \omega_{D}\right)$ of BiHom-coassociative coalgebras is a linear map $f: C \rightarrow D$ such that $\nu_{D} \circ f=f\circ \nu_{C}, \omega_{D} \circ f=f \circ \omega_{C}$ and $(f \otimes f) \circ \Delta_{C}=\Delta_{D} \circ f$.

\begin{definition}
An BiHom-Frobenius algebra is a 7-tuple $(H, \cdot, \Delta, \alpha, \beta, \nu, \omega)$, where $(H, \cdot, \alpha, \beta)$ is a BiHom-associative algebra, $(H, \Delta, \nu, \omega)$ is a BiHom-coassociative coalgebra and the following compatibility conditions are satisfied, for all $x, y \in H$:
\begin{eqnarray}\label{def:bihom-inf-01}
&&\Delta \circ \mu=(\mu \otimes \beta) \circ(\omega \otimes \Delta)=(\alpha \otimes \mu) \circ(\Delta \otimes \nu), \\
&&\alpha \circ \nu=\nu \circ \alpha, \quad \alpha \circ \omega=\omega \circ \alpha, \quad \beta \circ \nu=\nu \circ \beta, \quad \beta \circ \omega=\omega \circ \beta, \\
&&(\alpha \otimes \alpha) \circ \Delta=\Delta \circ \alpha, \quad(\beta \otimes \beta) \circ \Delta=\Delta \circ \beta, \\
&&\nu(xy)=\nu(x) \nu(y), \quad \omega(xy)=\omega(x)  \omega(y).
\end{eqnarray}
\end{definition}

In terms of elements, the above condition \eqref{def:bihom-inf-01} can be rewritten as
\begin{eqnarray}\label{def:bihom-inf-02}
\Delta(xy)=\omega(x)  y_{1} \otimes \beta\left(y_{2}\right)=\alpha\left(x_{1}\right) \otimes x_{2}  \nu(y)
\end{eqnarray}
where we use the Sweedler notation $\Delta(x):= x_{1} \otimes x_{2}$.


When $\nu=\beta,\  \omega=\alpha$, we obtain the BiHom-Frobenius algebra satisfying the compatibility condition:
\begin{eqnarray}
\Delta(xy)=(\mu\ot \beta)(\alpha(x\cdot\Delta(y))=(\alpha\ot \mu)(\Delta(x)\cdot\beta(y)),
\end{eqnarray}
or, in the Sweedler notation
\begin{eqnarray}\label{def:bihom-inf-04}
\Delta(xy)
=\alpha(x) y_{1} \otimes \beta\left(y_{2}\right)=\alpha\left(x_{1}\right) \otimes x_{2} \beta(y).
\end{eqnarray}

In the following of this paper, we assume $\nu=\beta, \omega=\alpha$ and denote this type of BiHom-Frobenius algebra by  $(A,\cdot, \Delta, \alpha, \beta)$ or simply by $(A,  \alpha, \beta)$.

For a BiHom-algebra $(H,\alhh,\behh)$ and a BiHom-vector space $(V,\alvv,\bevv)$ if there is a linear map $\trr: H\ot V \to V, x\ot v\mapsto x\trr v$ such that
\begin{equation}
(xy)\trr \bevv(v)=\alhh(x)\trr(y\trr v),
 \end{equation}
for all $x, y\in H, v\in V$, then $(V, \trr )$ is called a left $H$-BiHom-module. For a BiHom-coalgebra $(H,\alhh,\behh)$, if there is a
linear map $\phi :V\to H\ot V, \phi(v)=v\moi\ot v\mo$ such that
\begin{equation}
\Delta_H(v\loi)\ot \bevv(v\loo)= \alhh(v\loi)\ot \phi(v\loo),
 \end{equation}
then $(V, \phi)$ is called a left $H$-BiHom-comodule.
If $(A,\alaa,\beaa)$ and $(H,\alhh,\behh)$ are   BiHom-algebras, $(A,\alaa,\beaa)$ is a left $H$-BiHom-module and
\begin{equation}
\alhh(x)\trr (ab)=(x\trr a)\beaa(b),
 \end{equation}
then $(A, \cdot, \trr)$ is called a left $H$-BiHom-module  algebra.
 If $(H,\alhh,\behh)$ is a BiHom-coalgebra and $(A,\alaa,\beaa)$  is a    BiHom-algebra,  $(A,\alaa,\beaa)$  is a left $H$-BiHom-comodule and
\begin{equation}
\phi(ab)=\alhh(a\lmoi) \ot a\lmoo \beaa(b),
 \end{equation}
then $(A,\alaa,\beaa)$  is called a left $H$-BiHom-comodule algebra.
The right BiHom-module algebra and right BiHom-comodule algebra are defined similarly.

\begin{definition}
Let $({H},\alhh,\behh)$ be a BiHom-algebra, $(V,\alvv,\bevv)$ a BiHom-vector space. Then $(V,\alvv,\bevv)$ is called an ${H}$-BiHom-bimodule if there is a pair of linear maps $ \trr :{H}\otimes V \to V,(x,v) \to x \trr v$ and $\trl :V\otimes {H} \to V,(v,x) \to v \trl x$  such that the following conditions hold:
\begin{eqnarray}
  &&(xy) \trr \bevv(v) = \alhh(x) \trr (y \trr v),\\
  && (v \trl x) \trl \behh(y) =\alvv(v)\trl(xy),\\
  &&(x\trr v)\trl \behh(y)  = \alhh(x)\trr (v \trl y),
\end{eqnarray}
for all $x,y\in {H}$ and $v\in V.$
\end{definition}
The category of  BiHom-bimodules over $(H,\alhh,\behh)$ is denoted  by ${}_{H}\mathcal{M}{}_{H}$.

\begin{definition}
Let $(H,\alhh,\behh)$ be a BiHom-coalgebra, $(V,\alvv,\bevv)$ a BiHom-vector space. Then $(V,\alvv,\bevv)$ is called an ${A}$-BiHom-bicomodule if there is a pair of linear maps $\phi:V\to {A}\otimes V, \phi(v)=v\moi\ot v\mo$ and $\psi:V\to V\otimes {A}, \psi(v)=v\mo\ot v\mi$  such that the following conditions hold:
\begin{eqnarray}
  &&\alaa(v_{(-1)}) \otimes \phi\left(v_{(0)}\right)=\Delta_{A}\left(v_{(-1)}\right) \otimes \bevv(v_{(0)}),\\
  &&\psi\left(v_{(0)}\right) \otimes \beaa(v_{(1)})=\alvv(v_{(0)}) \otimes \Delta_{A}\left(v_{(1)}\right),\\
    &&\phi\left(v_{(0)}\right) \otimes \beaa(v_{(1)})=\alaa(v_{(-1)}) \otimes \psi\left(v_{(0)}\right),
\end{eqnarray}
for all $v\in V.$
\end{definition}
The category of  BiHom-bicomodules over $(H,\alhh,\behh)$ is denoted by ${}^{H}\mathcal{M}{}^{H}$.

\begin{definition}
Let $({H},\alhh,\behh)$ and  $(A,\alaa,\beaa)$ be BiHom-algebras. An action of $({H},\alhh,\behh)$ on $(A,\alaa,\beaa)$ is a pair of linear maps $\trr:{H}\otimes {A} \to {A},(x, a) \to x \trr a$ and $\trl: {A}\otimes {H} \to {A}, (a, x) \to a \trl x$  such that $(A,\alaa,\beaa)$ is an $H$-BiHom-bimodule and  the following conditions hold:
\begin{eqnarray}
  &&\alhh(x)\trr (ab)=(x\trr a)\beaa(b),\\
  &&\alaa(a) (x\trr b)=(a\trl x)\beaa(b) ,\\
  &&\alaa(a) (b\trl x)=(ab)\trl \behh(x),
\end{eqnarray}
for all $x\in {H}$ and $a, b\in {A}.$ In this case, we call $(A,\trr,\trl,\alpha,\beta)$ to be an $H$-BiHom-bimodule algebra.
\end{definition}

\begin{definition}
Let $({H},\alhh,\behh)$ and  $(A,\alaa,\beaa)$ be BiHom-coalgebras. A coaction of $({H},\alhh,\behh)$ on $(A,\alaa,\beaa)$ is a pair of linear maps $\phi:{A}\to {H}\otimes {A}, \phi(a)=a\moi\ot a\mo$ and $\psi:{A}\to {A}\otimes {H}, \psi(a)=a\mo\ot a\mi$ such that $(A,\alaa,\beaa)$ is an $H$-BiHom-bicomodule and  the following conditions hold:
\begin{eqnarray}
  &&\alhh(a_{(-1)}) \otimes \Delta_{A}\left(a_{(0)}\right)=\phi\left(a_{1}\right) \otimes \beaa(a_{2}),\\
  &&\alaa(a_{1}) \otimes \psi\left(a_{2}\right)=\Delta_{A}\left(a_{(0)}\right) \otimes \behh(a_{(1)}),\\
    &&\alaa(a_{1}) \otimes \phi\left(a_{2}\right)=\psi\left(a_{1}\right) \otimes \beaa(a_{2}),
\end{eqnarray}
for all $a\in {A}.$ In this case, we call $(A,\phi,\psi,\alpha,\beta)$ to be an $H$-BiHom-bicomodule coalgebra.
\end{definition}

\section{Braided BiHom-Frobenius algebras}

\subsection{ BiHom-Frobenius-Hopf bimodules and braided BiHom-Frobenius algebras}
\begin{definition}Let $(H,\alhh,\behh)$ be a BiHom-Frobenius algebra.
 If $(V,\alvv,\bevv)$ is  a left $H$-BiHom-module and a left $H$-BiHom-comodule, and satisfying the following condition
\begin{enumerate}
\item[(HM1)] $\phi(x \trr v)=\alhh(x_{1}) \otimes\left(x_{2}\trr \bevv(v) \right)=(\alhh(x) v_{(-1)}) \otimes \bevv(v_{(0)})$,
\end{enumerate}
then $(V, \trr, \phi,\alpha_{V},\beta_{V})$ is called a left BiHom-Frobenius-Hopf module over $H$.
\end{definition}
The  category of left infinitesimal Hopf modules over $(H,\alhh,\behh)$ is denoted by ${}^{H}_{H}\mathcal{M}$.

\begin{definition}Let $(H,\alhh,\behh)$ be a BiHom-Frobenius algebra.
 If $(V,\alvv,\bevv)$ is  a right $H$-BiHom-module and a right $H$-BiHom-comodule, satisfying
 \begin{enumerate}
\item[(HM2)]  $\psi(v \trl x)=\left(\alvv(v)\trl x_{1}\right) \otimes \behh(x_{2})=\alvv(v_{(0)}) \otimes (v_{(1)} \behh(x)),$
\end{enumerate}
then $(V, \trl, \psi,\alpha_{V},\beta_{V})$ is called a right BiHom-Frobenius-Hopf module over $H$.
\end{definition}
The  category of right infinitesimal Hopf modules over $(H,\alhh,\behh)$ is denoted by $\mathcal{M}{}^{H}_{H}$.

\begin{definition}Let $(H,\alhh,\behh)$ be a BiHom-Frobenius algebra.
 If $(V,\alvv,\bevv)$ is simultaneously a BiHom-bimodule, a  BiHom-bicomodule, a left BiHom-Frobenius-Hopf module, a right BiHom-Frobenius-Hopf module over $(H,\alhh,\behh)$ and satisfying
 the following compatibility conditions
 \begin{enumerate}
\item[(HM3)]  $\psi(x \trr v)=\left(\alhh(x) \trr  v_{(0)}\right) \otimes \behh(v_{(1)})$,
\item[(HM4)]  $ \phi(v \trl x)=\alhh(v_{(-1)}) \otimes\left(v_{(0)} \trl \behh(x)\right).$
\end{enumerate}
then $(V,\trr,\trl,\alvv,\bevv)$ is called a BiHom-Frobenius-Hopf bimodule over $H$.
\end{definition}
We denote  the  category of  BiHom-Frobenius-Hopf bimodules over $H$ by ${}^{H}_{H}\mathcal{M}{}^{H}_{H}$.

The second  motivation of introducing this  category is a suitable place to define our braided BiHom-Frobenius algebras.

\begin{definition} Let $(H,\alhh,\behh)$  be a BiHom-Frobenius algebra.
If $(A,\alaa,\beaa)$ be a BiHom-algebra and a BiHom-coalgebra in ${}^{H}_{H}\mathcal{M}{}^{H}_{H}$, we call $(A,\alaa,\beaa)$ a \emph{braided BiHom-Frobenius algebra}, if the following condition is satisfied
\begin{enumerate}
\item[(BB1)]
$\Delta_{A}(a b)=\alaa(a_{1}) \otimes a_{2} \beaa(b)+\alaa(a_{(0)}) \otimes\left(a_{(1)} \trr \beaa(b)\right)$

$\qquad\quad\ =\alaa(a) b_{1} \otimes \beaa(b_{2})+(\alaa(a)\trl b_{(-1)}) \otimes \beaa(b_{(0)}).$
\end{enumerate}
\end{definition}

Here $(A,\alaa,\beaa)$ is a BiHom-algebra and a BiHom-coalgebra in ${}^{H}_{H}\mathcal{M}{}^{H}_{H}$ means that $(A,\alaa,\beaa)$ is simultaneously an $H$-BiHom-bimodule algebra (coalgebra) and an $H$-BiHom-bicomodule algebra (coalgebra).

Now we construction BiHom-Frobenius algebra from braided BiHom-Frobenius algebra.
Let $(H,\alhh,\behh)$  be a BiHom-Frobenius algebra, $(A,\alaa,\beaa)$ be a BiHom-algebra and a BiHom-coalgebra in ${}^{H}_{H}\mathcal{M}{}^{H}_{H}$.
We define the Bi-Hom maps, multiplication and comultiplication on the direct sum vector space $E=A \oplus H$ by
$$
\begin{aligned}
&{\alpha_E}(a, x):=(\alpha_A(a), \alpha_H(x)),\quad{\beta_E}(a, x):=(\beta_A(a), \beta_H(x)),\\
&(a, x)(b, y):=(a b+x\trr b+a \trl y, x y), \\
&\Delta_{E}(a, x):=\Delta_{A}(a)+\phi(a)+\psi(a)+\Delta_{H}(x).
\end{aligned}
$$
This is called the biproduct of $(A,\alaa,\beaa)$ and $({H},\alhh,\behh)$ which will be  denoted   by $A\lbiprod H$.

\begin{theorem} Let  $(H,\alhh,\behh)$ be a BiHom-Frobenius algebra.
Then the biproduct $A\lbiprod H$ forms a BiHom-Frobenius algebra if and only if  $(A,\alaa,\beaa)$ is a braided BiHom-Frobenius algebra in ${}^{H}_{H}\mathcal{M}{}^{H}_{H}$.
\end{theorem}

\begin{proof}
First, we need to prove the multiplication defined above is BiHom-associative. For $\forall a, b, c\in A$, $\forall x, y, z\in H$, we will check that
$((a, x) (b, y)) {\beta_E}(c, z)={\alpha_E}(a, x)((b, y) (c, z))$.
By definition, the left hand side is equal to
$$\begin{aligned}
&\big((a, x) (b, y)\big) {\beta_E}(c, z)\\
=&\big(a b+x \trr b+a \trl y, x y\big) (\beaa(c), \behh(z))\\
=&\big((a b) \beaa(c)+(x \trr b) \beaa(c)+(a \trl y) \beaa(c)\\
&+(x y)\trr \beaa(c)+(ab)\trl \behh(z)+(x \trr b) \trl \behh(z)+(a \trl y) \trl \behh(z),(x y) \behh(z)\big)
\end{aligned}
$$
and the right hand side is equal to
$$\begin{aligned}
&{\alpha_E}(a, x)\big((b, y)(c, z)\big)\\
=&(\alaa(a), \alhh(x))(b c+y\trr c+b \trl z, y z)\\
=&\big((\alaa(a)(b c)+\alaa(a)(y\trr c)+\alaa(a)(b \trl z)\\
&+\alhh(x) \trr (b c)+\alhh(x) \trr (y\trr c)+\alhh(x)\trr (b \trl z)+\alaa(a) \trl(y z), \alhh(x)(y z)\big).
\end{aligned}
$$
Thus the two sides are equal to each other if and only if $(A, \trr, \trl,\alpha_{A},\beta_{A})$ is a BiHom-bimodule algebra over $H$.

Next, we need to prove the comultiplication is BiHom-coassociative. For any $(a, x)\in A \oplus H$, we need to prove $({\alpha_E} \otimes \Delta) \Delta_E(a, x)=(\Delta \otimes {\beta_E}) \Delta_E(a, x)$. By definition, the left hand side is equal to
$$
\begin{aligned}
&({\alpha_E} \otimes \Delta) \Delta_E(a, x)\\
=&({\alpha_E} \otimes \Delta)\left(a_{1} \otimes a_{2}+a_{(-1)} \otimes a_{(0)}+a_{(0)} \otimes a_{(1)}+x_{1} \otimes x_{2}\right)\\
=&\alaa(a_{1}) \otimes \Delta_{A}\left(a_{2}\right)+\alaa(a_{1}) \otimes \phi\left(a_{2}\right)+\alaa(a_{1}) \otimes \psi\left(a_{2}\right)\\
&+\alhh(a_{(-1)}) \otimes \Delta_{A}\left(a_{(0)}\right)+\alhh(a_{(-1)}) \otimes \phi\left(a_{(0)}\right)+\alhh(a_{(-1)}) \otimes \psi\left(a_{(0)}\right)\\
&+\alaa(a_{(0)}) \otimes \Delta_{H}\left(a_{(1)}\right)+\alhh(x_{1}) \otimes \Delta_{H}\left(x_{2}\right)
\end{aligned}
$$
and the right hand side is equal to
$$
\begin{aligned}
&(\Delta \otimes {\beta_E}) \Delta_E(a, x)\\
=&(\Delta\otimes {\beta_E})\left(a_{1} \otimes a_{2}+a_{(-1)} \otimes a_{(0)}+a_{(0)} \otimes a_{(1)}+x_{1} \otimes x_{2}\right)\\
=&\Delta_{A}\left(a_{1}\right) \otimes \beaa(a_{2})+\phi\left(a_{1}\right) \otimes \beaa(a_{2})+\psi\left(a_{1}\right) \otimes \beaa(a_{2})
+\Delta_{H}\left(a_{(-1)}\right) \otimes \beaa(a_{(0)})\\
&+\Delta_{A}\left(a_{(0)}\right) \otimes \behh(a_{(1)})+\phi\left(a_{(0)}\right) \otimes \behh(a_{(1)})
+\psi\left(a_{(0)}\right) \otimes \behh(a_{(1)})+\Delta_{H}\left(x_{1}\right) \otimes \behh(x_{2}).
\end{aligned}
$$
Thus the two sides are equal to each other if and only if $(A, \phi, \psi,\alpha_{A},\beta_{A})$ is a BiHom-bicomodule coalgebra over $H$.

 Finally, we show the compatibility condition:
$$\Delta_E((a, x) (b, y))=({\alpha_E} \otimes \mu_E)(\Delta_E(a, x) \cdot{\beta_E}(b, y))=(\mu_E \otimes {\beta_E})({\alpha_E}(a, x) \cdot \Delta(b, y)).$$

By direct computations, the left hand side is equal to
$$\begin{aligned}
&\Delta_E((a, x) (b, y))\\
=&\Delta_E(ab+x \trr b+a \trl y, x y)\\
=&\Delta_A(a b)+\phi(a b)+\psi(a b)+\Delta_A(x \trr b)+\phi(x \trr b)+\psi(x \trr b)\\
&+\Delta_A(a \trl y)+\phi(a \trl y)+\psi(a \trl y)+\Delta_{H}(x y),
\end{aligned}
$$
the middle and right hand sides are equal to
$$\begin{aligned}
&({\alpha_E} \otimes \mu_E)(\Delta_E(a, x) \cdot{\beta_E}(b, y))\\
=&({\alpha_E} \otimes \mu_E)\left(\left(a_{1} \otimes a_{2}+a_{(-1)} \otimes a_{(0)}+a_{(0)} \otimes a_{(1)}+x_{1} \otimes x_{2}\right) \cdot(\beaa(b), \behh(y))\right)\\
=&\alaa(a_{1}) \otimes\left(a_{2} \beaa(b)+a_{2} \trl \behh(y)\right)+\alhh(a_{(-1)}) \otimes\left(a_{(0)} \beaa(b)+a_{(0)} \trl \behh(y)\right)\\
&+\alaa(a_{(0)}) \otimes\left(a_{(1)} \trr \beaa(b)\right)+\alaa(a_{(0)}) \otimes\left(a_{(1)} \behh(y)\right)\\
&+\alhh(x_{1}) \otimes\left(x_{2} \trr \beaa(b)\right)+\alhh(x_{1}) \otimes\left(x_{2} \behh(y)\right),
\end{aligned}
$$
$$\begin{aligned}
&(\mu_E \otimes {\beta_E})({\alpha_E}(a, x) \cdot \Delta(b, y))\\
=&(\mu_E \otimes {\beta_E})\left((\alaa(a), \alhh(x)) \cdot\left(b_{1} \otimes b_{2}+b_{(-1)} \otimes b_{(0)}+b_{(0)} \otimes b_{(1)}+y_{1} \otimes y_{2}\right)\right)\\
=&+\left(\alaa(a) b_{1}+\alhh(x) \trr b_{1}\right) \otimes \beaa(b_{2})+(\alaa(a) \trl b_{(-1)}) \otimes \beaa(b_{(0)})\\
&+\left(\alhh(x) b_{(-1)}\right) \otimes \beaa(b_{(0)})+\left(\alaa(a) b_{(0)}+\alhh(x)\trr b_{(0)}\right) \otimes \behh(b_{(1)})\\
&+\left(\alaa(a)\trl y_{1}\right) \otimes \behh(y_{2})+\left(\alhh(x) y_{1}\right) \otimes \behh(y_{2}).
\end{aligned}
$$
Then the two sides are equal to each other if and only if
\begin{enumerate}
\item[]
(1) $\Delta_{A}(a b)=\alaa(a_{1}) \otimes a_{2} \beaa(b)+\alaa(a_{(0)}) \otimes\left(a_{(1)} \trr \beaa(b)\right)$

$\qquad\qquad\quad=\alaa(a) b_{1} \otimes \beaa(b_{2})+(\alaa(a)\trl b_{(-1)}) \otimes \beaa(b_{(0)})$,

(2)$\phi(x \trr b)=\alhh(x_{1}) \otimes\left(x_{2}\trr \beaa(b) \right)=(\alhh(x) b_{(-1)}) \otimes \beaa(b_{(0)})$,

(3) $\psi(a \trl y)=\left(\alaa(a)\trl y_{1}\right) \otimes \behh(y_{2})=\alaa(a_{(0)}) \otimes (a_{(1)} \behh(y))$,

(4) $\psi(x \trr b)=\left(\alhh(x) \trr  b_{(0)}\right) \otimes \behh(b_{(1)})$,

(5) $\phi(a \trl y)=\alhh(a_{(-1)}) \otimes\left(a_{(0)} \trl \behh(y)\right)$,

(6) $\phi(a b)=\alhh(a_{(-1)}) \otimes\left(a_{(0)} \beaa(b)\right) $,

(7) $\psi(a b)=(\alaa(a) b_{(0)}) \otimes \behh(b_{(1)})$,

(8) $\Delta_{A}(x\trr b)=\left(\alhh(x) \trr b_{1}\right) \otimes \beaa(b_{2})$,

(9) $\Delta_{A}(a \trl y)=\alaa(a_{1}) \otimes\left(a_{2} \trl \behh(y)\right)$.
\end{enumerate}
From (6)--(9) we have that $(A,\alaa,\beaa)$ is a left and right $H$-BiHom-module coalgebra and $H$-BiHom-comodule algebra,
from (2)--(5) we get that $(A,\alaa,\beaa)$ is a  BiHom-Frobenius-Hopf bimodule over $H$, and (1) is the condition for $(A,\alaa,\beaa)$ to be a braided BiHom-Frobenius algebra.
The proof is completed.
\end{proof}

\section{Cocycle bicrossproduct of braided BiHom-Frobenius algebras}
The construction of cocycle cross product bialgebras was studied by Y. Bespalov and B. Drabant in \cite{BD99,BD01}.
In this section, we give a BiHom-Frobenius algebras version.
\subsection{Matched pair of braided BiHom-Frobenius algebras}
In this subsection, we introduce the notion of matched pairs of BiHom-algebras, BiHom-coalgebras and BiHom-Frobenius algebras.
Let $A, H$ be both  BiHom-algebras and   BiHom-coalgebras.   For $a, b\in A$, $x, y\in H$,  we denote maps
\begin{align*}
&\ppr: H \otimes A \to A,\quad \ppl: A\otimes H\to A,\\
&\trr: A\otimes H \to H,\quad \trl: H\otimes A \to H,\\
&\phi: A \to H \otimes A,\quad  \psi: A \to A\otimes H,\\
&\rho: H  \to A\otimes H,\quad  \gamma: H \to H \otimes A,
\end{align*}
by
\begin{eqnarray*}
&& \ppr (x \otimes a) =x \ppr a, \quad \ppl(a\otimes x) = a \ppl x, \\
&& \trr (a \otimes x) = a \trr x, \quad \trl(x \otimes a) = x \triangleleft a, \\
&& \phi (a)=\sum a\lmoi\ot a\loo, \quad \psi (a) = \sum a\loo \ot a\lmi,\\
&& \rho (x)=\sum x\boi\ot x\boo, \quad \gamma (x) = \sum x\boo \ot x\bi.
\end{eqnarray*}

\begin{definition}
A \emph{matched pair} of   BiHom-algebras is a system $(A, \, {H},\, \trl, \, \trr, \, \ppl, \, \ppr)$ consisting
of two   BiHom-algebras $(A,\alaa,\beaa)$ and $({H},\alhh,\behh)$ and four bilinear maps $\triangleleft : {H}\otimes A\to {H}$, $\trr : {A} \otimes H
\to H$, $\ppl:A \otimes {H} \to A$, $\ppr: H\otimes {A} \to {A}$ such that $({H},\trr,\trl)$ is an $A$-BiHom-bimodule,  $(A,\ppr,\ppl)$ is an ${H}$-BiHom-bimodule and satisfying the following compatibilities for all $a, b\in A$, $x, y \in {H}$:
\begin{enumerate}
\item[(M1)] $\alhh(x)\ppr ( ab)=(x\ppr a) \beaa(b)+ (x\trl a)\ppr \beaa(b)$,

\item[(M2)] $( ab) \ppl \behh(x)=\alaa(a)(b \ppl x) + \alaa(a) \ppl ( b \trr x)$,

\item[(M3)]$\alaa(a) \trr (x y)=(a \ppl x) \trr \behh(y) + (a \trr x)  \behh(y)$,

\item[(M4)] $ (xy) \trl \beaa(a)= \alhh(x)\trl (y\ppr a) + \alhh(x) ( y\trl a)$,

\item[(M5)] $  \alaa(a)(x\ppr b)+ \alaa(a) \ppl (x\trl b)=(a \ppl x)\beaa( b)+ (a \trr x)\ppr \beaa(b) $,

\item[(M6)] $\alhh(x)\trl (a \ppl y)+ \alhh(x)   ( a \trr y)=(x\ppr a) \trr \behh(y) + (x\trl a)   \behh(y)$.
\end{enumerate}
\end{definition}

\begin{proposition}\cite{HHS}
Let  $(A,\alaa,\beaa)$ and $({H},\alhh,\behh)$ be a matched pair of   BiHom-algebras.
Then $E=A \, \bowtie {H}= A \oplus  {H}$, as a vector space, with the multiplication defined for any $a, b\in A$ and $x, y\in {H}$ by
\begin{eqnarray*}
&&{\alpha_E}(a, x):=(\alpha_A(a), \alpha_H(x)),\quad{\beta_E}(a, x):=(\beta_A(a), \beta_H(x)),\\
&&(a, x) (b, y) := \big(ab+ a \ppl y + x\ppr b,\, \, a\trr y + x\trl b + xy \big)
\end{eqnarray*}
is a BiHom-associative algebra called the \emph{bicrossed product} associated to the matched pair of BiHom-algebras $(A,\alaa,\beaa)$ and $({H},\alhh,\behh)$.
\end{proposition}

%

\begin{definition} A \emph{matched pair} of   BiHom-coalgebras is a system $(A, \, {H}, \, \phi, \, \psi, \, \rho, \, \gamma)$ consisting
of two    BiHom-coalgebras $(A,\alaa,\beaa)$ and $({H},\alhh,\behh)$ and four bilinear maps
$\phi: {A}\to H\otimes A$, $\psi: {A}\to A \otimes H$, $\rho: H\to A\otimes {H}$, $\gamma: H \to {H} \ot {A}$
such that $({H},\rho, \gamma)$ is an $A$-BiHom-bicomodule,  $(A,\phi, \, \psi)$ is an ${H}$-BiHom-bicomodule and satisfying the following compatibility conditions for any $a\in A$, $x\in {H}$:
\begin{enumerate}
\item[(MC1)]
$\alhh(a_{(-1)}) \otimes \Delta_{A}\left(a_{(0)}\right)=\phi\left(a_{1}\right) \otimes \beaa(a_{2})+\gamma\left(a_{(-1)}\right) \otimes \beaa(a_{(0)})$,

\item[(MC2)]
$\Delta_{A}\left(a_{(0)}\right) \otimes \behh(a_{(1)})=\alaa(a_{1}) \otimes \psi\left(a_{2}\right)+\alaa(a_{(0)}) \otimes \rho(a_{(1)})$,

\item[(MC3)]
$\alaa(x_{[-1]}) \otimes \Delta_{H}\left(x_{[0]}\right)=\rho\left(x_{1}\right) \otimes \behh(x_{2})+\psi\left(x_{[-1]}\right) \otimes \behh(x_{[0]})$,

\item[(MC4)]
$ \Delta_{H}\left(x_{[0]}\right) \otimes \beaa(x_{[1]})=\alhh(x_{[0]}) \otimes \phi\left(x_{[1]}\right)+\alhh(x_{1}) \otimes \gamma\left(x_{2}\right)$,

\item[(MC5)]
$\alaa(a_{1}) \otimes \phi\left(a_{2}\right)+\alaa(a_{(0)}) \otimes \gamma\left(a_{(1)}\right)=\psi\left(a_{1}\right) \otimes \beaa(a_{2})+\rho\left(a_{(-1)}\right) \otimes \beaa(a_{(0)})$,

\item[(MC6)]
$ \alhh(x_{1}) \otimes \rho\left(x_{2}\right)+\alhh(x_{[0]}) \otimes \psi\left(x_{[1]}\right)=\phi\left(x_{[-1]}\right) \otimes \behh(x_{[0]})+\gamma\left(x_{1}\right) \otimes \behh(x_{2})$.
\end{enumerate}
\end{definition}

\begin{lemma}\label{lem1} Let $(A, H)$ be a matched pair of   BiHom-coalgebras. We define $E=A\lrcoprod H$ as the vector space $A\oplus H$ with   comultiplication
$$\Delta_{E}(a, x)=(\Delta_{A}+\phi+\psi)(a)+(\Delta_{H}+\rho+\gamma)(x),$$
that is
$$\Delta_{E}(a)=\sum a\li \ot a\lii+\sum a\loi \ot a\loo+\sum a\mo\ot a\mi, $$
$$\Delta_{E}(x)=\sum x\li \ot x\lii+\sum  x\boi \ot x\boo+\sum x\boo \ot x\bi.$$
Then  $A\lrcoprod H$ is a BiHom-coalgebra which is called the \emph{bicrossed coproduct} associated to the matched pair of   BiHom-coalgebras $(A,\alaa,\beaa)$ and $({H},\alhh,\behh)$ .
\end{lemma}

\begin{proof} The proof is by direct computations.
$$
\begin{aligned}
&({\alpha_E} \otimes \Delta_{E}) \Delta_{E}(a, x)\\
=&({\alpha_E} \otimes \Delta_{E})\left(a_{1} \otimes a_{2}+a_{(-1)} \otimes a_{(0)}+a_{(0)} \otimes a_{(1)}+x_{1} \otimes x_{2}+x_{[-1]} \otimes x_{[0]}+x_{[0]} \otimes x_{[1]}\right)\\
=&\alaa(a_{1}) \otimes \Delta_{A}\left(a_{2}\right)+\alaa(a_{1}) \otimes \phi\left(a_{2}\right)+\alaa(a_{1})\otimes \psi\left(a_{2}\right)\\
&+\alhh(a_{(-1)}) \otimes \Delta_{A}\left(a_{(0)}\right)+\alhh(a_{(-1)}) \otimes \phi\left(a_{(0)}\right)+\alhh(a_{(-1)}) \otimes \psi\left(a_{(0)}\right)\\
&+\alaa(a_{(0)}) \otimes \Delta_{H}\left(a_{(1)}\right)+\alaa(a_{(0)}) \otimes \rho\left(a_{(1)}\right)+\alaa(a_{(0)}) \otimes \gamma\left(a_{(1)}\right)\\
&+\alhh(x_{1}) \otimes \Delta_{H}\left(x_{2}\right)+\alhh(x_{1}) \otimes \rho\left(x_{2}\right)+\alhh(x_{1}) \otimes \gamma\left(x_{2}\right)\\
&+\alaa(x_{[-1]}) \otimes \Delta_{H}\left(x_{[0]}\right)+\alaa(x_{[-1]}) \otimes \rho\left(x_{[0]}\right)+\alaa(x_{[-1]}) \otimes \gamma\left(x_{[0]}\right)\\
&+\alhh(x_{[0]}) \otimes \Delta_{A}\left(x_{[1]}\right)+\alhh(x_{[0]}) \otimes \phi\left(x_{[1]}\right)+\alhh(x_{[0]}) \otimes \psi\left(x_{[1]}\right)
\end{aligned}
$$
$$
\begin{aligned}
& (\Delta_{E} \otimes {\beta_E}) \Delta_{E}(a, x)\\
=&(\Delta\otimes {\beta_E})\left(a_{1} \otimes a_{2}+a_{(-1)} \otimes a_{(0)}+a_{(0)} \otimes a_{(1)}+x_{1} \otimes x_{2}+x_{[-1]} \otimes x_{[0]}+x_{[0]} \otimes x_{[1]}\right)\\
=&\Delta_{A}\left(a_{1}\right) \otimes \beaa(a_{2})+\phi\left(a_{1}\right) \otimes \beaa(a_{2})+\psi\left(a_{1}\right) \otimes \beaa(a_{2})\\
&+\Delta_{H}\left(a_{(-1)}\right) \otimes \beaa(a_{(0)})+\rho\left(a_{(-1)}\right) \otimes \beaa(a_{(0)})+\gamma\left(a_{(-1)}\right) \otimes \beaa(a_{(0)})\\
&+\Delta_{A}\left(a_{(0)}\right) \otimes \behh(a_{(1)})+\phi\left(a_{(0)}\right) \otimes \behh(a_{(1)})+\psi\left(a_{(0)}\right) \otimes \behh(a_{(1)})\\
&+\Delta_{H}\left(x_{1}\right) \otimes \behh(x_{2})+\rho\left(x_{1}\right) \otimes \behh(x_{2})+\gamma\left(x_{1}\right) \otimes \behh(x_{2})\\
&+\Delta_{A}\left(x_{[-1])}\right) \otimes \behh(x_{[0]})+\phi\left(x_{[-1]}\right) \otimes \behh(x_{[0]})+\psi\left(x_{[-1]}\right) \otimes \behh(x_{[0]})\\
&+\Delta_{H}\left(x_{[0]}\right) \otimes \beaa(x_{[1]})+\rho\left(x_{[0]}\right) \otimes \beaa(x_{[1]})+\gamma\left(x_{[0]}\right) \otimes \beaa(x_{[1]})
\end{aligned}
$$
Thus the comultiplication  is coassociative if and only if (MC1)--(MC6) hold.
\end{proof}

In the following of this section, we construct BiHom-Frobenius algebra from the double cross biproduct of a pair of braided BiHom-Frobenius algebras.

\begin{definition}
Let $(A,\alaa,\beaa)$ be simultaneously a  BiHom-algebra and a  BiHom-coalgebra.
 If $(V,\alvv,\bevv)$ is  a left $A$-BiHom-module and left $A$-BiHom-comodule, satisfying
\begin{align}
&\rho (a\trr v)= \alaa(a) v\boi\ot \bevv(v\boo) = \alaa(a\li) \ot (a\lii\trr \bevv(v)),
\end{align}
\noindent then $V$ is called a left BiHom-Frobenius-Hopf module over $A$.
\end{definition}
We denote  the  category of  left BiHom-Frobenius-Hopf modules over $(A,\alaa,\beaa)$ by ${}^{A}_{A}\mathcal{M}$.

\begin{definition}
Let $(A,\alaa,\beaa)$ be simultaneously a  BiHom-algebra and a BiHom-coalgebra.
 If $(V,\alvv,\bevv)$ is  a right $A$-BiHom-module and a right $A$-BiHom-comodule, satisfying
\begin{align}
&\gamma (v\trl a) = \alvv(v\boo)\ot (v\bi \beaa(a)) =\alvv(v)\trl a\li\ot \beaa(a\lii),
\end{align}
\noindent then $(V,\alvv,\bevv)$ is called a right BiHom-Frobenius-Hopf module over $A$.
\end{definition}
We denote  the  category of right BiHom-Frobenius-Hopf modules over $(A,\alaa,\beaa)$ by
$\mathcal{M}{}^{A}_{A}$.

\begin{definition}
Let $(A,\alaa,\beaa)$ be simultaneously a BiHom-algebra and a BiHom-coalgebra.
 If $(V,\alvv,\bevv)$ is simultaneously a BiHom-bimodule, a  BiHom-bicomodule, a left and  right BiHom-Frobenius-Hopf module over $(A,\alaa,\beaa)$ and satisfying
 the following compatibility conditions
\begin{align}
&\gamma (a \trr v)=\left(\alaa(a) \trr  v\boo\right) \otimes \beaa(v\bi),\quad \rho(v \trl a)=\alaa(v\boi) \otimes\left( v\boo \trl \beaa(a)\right).
\end{align}
then $(V,\alvv,\bevv)$ is called a BiHom-Frobenius-Hopf bimodule over $A$.
\end{definition}
We denote  the  category of  BiHom-Frobenius-Hopf bimodules over $(A,\alaa,\beaa)$ by ${}^{A}_{A}\mathcal{M}{}^{A}_{A}$.

\begin{definition}
If $(A,\alaa,\beaa)$ be a BiHom-algebra and  a BiHom-coalgebra and  $(H,\alhh,\behh)$ is a BiHom-Frobenius-Hopf bimodule over $A$, we call $(H,\alhh,\behh)$ a \emph{braided BiHom-Frobenius algebra} in ${}^{A}_{A}\mathcal{M}^{A}_{A}$, if the following condition is satisfied:
\begin{enumerate}
\item[(BB2)]
$\Delta_{H}(x y)=\alhh(x_{1}) \otimes x_{2} \behh(y)+\alhh(x) y_{1} \otimes \behh(y_{2})+\alhh(x_{[0]}) \otimes\left(x_{[1]} \trr \behh(y)\right)\\
+\left(\alhh(x) \trl y_{[-1]}\right) \otimes \behh(y_{[0]}).$
\end{enumerate}
\end{definition}

\begin{definition}\label{dmp}
Let $A, H$ be both   BiHom-algebras and   BiHom-coalgebras. If  the following conditions hold:
\begin{enumerate}
\item[(DM1)]  $\phi(a b)=\alhh(a_{(-1)}) \otimes\left(a_{(0)} \beaa(b)\right)=(\alaa(a) \trr b_{(-1)}) \otimes \beaa(b_{(0)})$,
\item[(DM2)] $\psi(a b)=(\alaa(a) b_{(0)}) \otimes \behh(b_{(1)})=\alaa(a_{(0)}) \otimes\left(a_{(1)} \trl \beaa(b)\right)$,
\item[(DM3)] $\rho(x y)=\alaa(x_{[-1]}) \otimes\left(x_{[0]} \behh(y)\right)=\left(\alhh(x) \ppr y_{[-1]}\right) \otimes \behh(y_{[0]})$,
\item[(DM4)] $\gamma(x y)=\alhh(x_{[0]})\otimes (x_{[1]}\ppl \behh(y))=\alhh(x)y_{[0]}\otimes \beaa(y_{[1]})$,
\item[(DM5)] $\Delta_{A}(x \ppr b)=\alaa(x_{[-1]}) \otimes\left(x_{[0]} \ppr \beaa(b)\right)=\left(\alhh(x) \ppr b_{1}\right) \otimes \beaa(b_{2})$,
\item[(DM6)] $\Delta_{A}(a\ppl y)=\alaa(a_{1}) \otimes\left(a_{2} \ppl \behh(y)\right)=\left(\alaa(a)\ppl y_{[0]}\right) \otimes \beaa(y_{[1]})$,
\item[(DM7)] $\Delta_{H}(a \trr y)=\alhh(a_{(-1)}) \otimes\left(a_{(0)}\trr \behh(y)\right)=\left(\alaa(a) \trr y_{1}\right) \otimes \behh(y_{2})$,
\item[(DM8)] $\Delta_{H}(x \trl b)=\alhh(x_{1}) \otimes\left(x_{2}  \trl \beaa(b)\right)=\left(\alhh(x)\trl b_{(0)}\right) \otimes \behh(b_{(1)})$,
\item[(DM9)]
$\phi(x \ppr b)+\gamma(x\trl b)$\\
$=\alhh (x_{1}) \otimes\left(x_{2} \ppr \beaa(b)\right)+\alhh(x) b_{(-1)} \otimes \beaa(b_{(0)})$\\
$=\alhh(x_{[0]}) \otimes x_{[1]}\beaa( b)+\left(\alhh(x)\trl b_{1}\right) \otimes \beaa(b_{2})$,
\item[(DM10)]
$\psi(a\ppl y)+\rho(a \trr y)$\\
$=\alaa(a_{(0)}) \otimes a_{(1)}\behh( y)+\left(\alaa(a)\ppl y_{1}\right) \otimes \behh(y_{2})$\\
$=\alaa(a_{1}) \otimes\left(a_{2} \trr \behh(y)\right)+\alaa(a) y_{[-1]} \otimes \behh(y_{[0]})$,
\item[(DM11)]
$\psi(x \ppr b)+\rho(x\trl b)=(\alhh(x) \ppr b_{(0)}) \otimes \behh(b_{(1)})=\alhh(x_{[-1]}) \otimes\left(x_{[0]} \trl \beaa(b)\right)$,
\item[(DM12)]
$\phi(a\ppl y)+\gamma(a \trr y)=\alhh(a_{(-1)}) \otimes\left(a_{(0)}\ppl \behh(y)\right)=\left(\alaa(a)\trr y_{[0]}\right)\otimes \beaa(y_{[1]})$.
\end{enumerate}
\noindent then $(A, H)$ is called a \emph{double matched pair}.
\end{definition}

\begin{theorem}\label{main1}Let $(A, H)$ be matched pair of   BiHom-algebras and   BiHom-coalgebras,
$(A, H)$ be a double matched pair, $(A,\alaa,\beaa)$ is  a  braided BiHom-Frobenius algebra in
${}^{H}_{H}\mathcal{M}^{H}_{H}$, $(H,\alhh,\behh)$ is  a braided BiHom-Frobenius algebra in
${}^{A}_{A}\mathcal{M}^{A}_{A}$. If we define the double cross biproduct of $A$ and
$H$, denoted by $E=A\lrbiprod H$, $A\lrbiprod H=A\bowtie H$ as
BiHom-algebra, $A\lrbiprod H=A\lrcoprod H$ as   BiHom-coalgebra, then
$A\lrbiprod H$ becomes a BiHom-Frobenius algebra.
\end{theorem}

\begin{proof}  We only need to check the compatibility condition
$$\Delta_E((a, x) (b, y))=({\alpha_E} \otimes \mu_E)(\Delta_E(a, x) \cdot{\beta_E}(b, y))=(\mu_E \otimes {\beta_E})({\alpha_E}(a, x) \cdot \Delta(b, y)).$$
The left hand side is equal to
$$
\begin{aligned}
&\Delta_E((a, x) (b, y))\\
=&\Delta_E(ab+x \ppr b+a\ppl y, x y+x \trl b+a\trr y)\\
=&\Delta_A(a b)+\phi(a b)+\psi(a b)+\Delta_A(x \ppr b)+\phi(x \ppr b)+\psi(x \ppr b)\\
&+\Delta_{A}(a\ppl y)+\phi(a\ppl y)+\psi(a\ppl y)+\Delta_{H}(x y)+\rho(x y)+\gamma(x y)\\
&+\Delta_{H}(x \trl b)+\rho(x \trl b)+\gamma(x \trl b)+\Delta_{H}(a \trr y)+\rho(a \trr y)+\gamma(a \trr y),
\end{aligned}
$$
the middle and right hand sides are equal to
$$
\begin{aligned}
&({\alpha_E} \otimes \mu_E)(\Delta_E(a, x) \cdot{\beta_E}(b, y))\\
=&\alaa(a_{1}) \otimes a_{2} \beaa(b)+\alaa(a_{1}) \otimes\left(a_{2}\ppl \behh(y)\right)+\alaa(a_{1}) \otimes\left(a_{2} \trr \behh(y)\right)\\
&+\alhh(a_{(-1)}) \otimes a_{(0)} \beaa(b)+\alhh(a_{(-1)}) \otimes\left(a_{(0)}\ppl \behh(y)\right)+\alhh(a_{(-1)}) \otimes\left(a_{(0)} \trr \behh(y)\right)\\
&+\alaa(a_{(0)})\otimes\left(a_{(1)} \ppr \beaa(b)\right)+\alaa(a_{(0)}) \otimes\left(a_{(1)} \trl \beaa(b)\right)+\alaa(a_{(0)}) \otimes a_{(1)}\behh(y)\\
&+\alhh(x_{1}) \otimes\left(x_{2}\ppr \beaa(b)\right)+\alhh(x_{1}) \otimes\left(x_{2} \trl \beaa(b)\right)+\alhh(x_{1}) \otimes x_{2} \behh(y) \\
&+\alaa(x\boi)\ot(x\boo\ppr \beaa(b))+\alaa(x\boi)\ot(x\boo\trl \beaa(b))+\alaa(x\boi)\ot x\boo \behh(y)\\
&+\alhh(x\boo)\ot x\bi \beaa(b)+\alhh(x\boo)\ot (x\bi\ppl \behh(y))+\alhh(x\boo)\ot (x\bi\trr \behh(y)),
\end{aligned}
$$
$$
\begin{aligned}
&(\mu_E \otimes {\beta_E})({\alpha_E}(a, x) \cdot \Delta(b, y))\\
=&\alaa(a) b_{1} \otimes \beaa(b_{2})+\left(\alhh(x) \ppr b_{1}\right) \otimes \beaa(b_{2})+\left(\alhh(x) \trl b_{1}\right) \otimes \beaa(b_{2})\\
&+(\alaa(a)\ppl b_{(-1)}) \otimes \beaa(b_{(0)})+(\alaa(a) \trr b_{(-1)}) \otimes \beaa(b_{(0)})+\alhh(x) b_{(-1)} \otimes \beaa(b_{(0)})\\
&+\alaa(a) b_{(0)} \otimes \behh(b_{(1)})+(\alhh(x) \ppr b_{(0)}) \otimes \behh(b_{(1)})+(\alhh(x) \trl b_{(0)}) \otimes \behh(b_{(1)})\\
&+\left(\alaa(a)\ppl y_{1}\right) \otimes \behh(y_{2})+\left(\alaa(a) \trr y_{1}\right) \otimes \behh(y_{2})+\alhh(x) y_{1}\ot \behh(y_{2})\\
&+\alaa(a)y\boi\ot \behh(y\boo)+(\alhh(x)\ppr y\boi)\ot \behh(y\boo)+(\alhh(x)\trl y\boi)\ot \behh(y\boo)\\
&+(\alaa(a)\ppl y\boo)\ot \beaa(y\bi)+(\alaa(a)\trr y\boo)\ot \beaa(y\bi)+\alhh(x)y\boo\ot \beaa(y\bi).
\end{aligned}
$$
Thus both sides are equal to each other if and only if  the double matched pair conditions (CDM1)--(CDM12) in Definition \ref{dmp} hold. The proof is completed.
\end{proof}

\subsection{Cocycle bicrossproduct BiHom-Frobenius algebras}

In this section, we construct cocycle bicrossproduct BiHom-Frobenius algebras, which is a generalization of double cross biproduct.

Let $A, H$ be both   algebras and   coalgebras.   For $a, b\in A$, $x, y\in H$,  we denote maps
\begin{align*}
&\sigma: H\otimes H \to A,\quad \theta: A\otimes A \to H,\\
&P: A  \to H\otimes H,\quad  Q: H \to A\otimes A,
\end{align*}
by
\begin{eqnarray*}
&& \sigma (x,y)  \in  A, \quad \theta(a, b) \in H,\\
&& P(a)=\sum a\ppi\ot  a\pii, \quad Q(x) = \sum x\qi \ot x\qii.
\end{eqnarray*}

A bilinear map $\si: H\ot H\to A$ is called a cocycle on $(H,\alhh,\behh)$ if
\begin{enumerate}
\item[(CC1)] $\sigma(x y, \behh(z))+\sigma(x, y)\ppl \behh(z)=\alhh(x) \ppr \sigma(y, z)+{\sigma}(\alhh(x), y z).$
\end{enumerate}

A bilinear map $\theta: A\ot A\to H$ is called a cocycle on $(A,\alaa,\beaa)$ if
\begin{enumerate}
\item[(CC2)] $\theta(a b, \beaa(c))+\theta(a, b) \triangleleft \beaa(c)=\alaa(a)\trr \theta(b, c)+\theta(\alaa(a), b c).$
\end{enumerate}

A bilinear map $P: A\to H\ot H$ is called a cycle on $(A,\alaa,\beaa)$ if
\begin{enumerate}
\item[(CC3)]  $\alhh(a\ppi)\ot \Delta(a\ppi)+\alhh(a\lmoi)\ot P(a\lmoo)\\
=\Delta_H(a\ppi)\ot \behh(a\pii)+P(a\lmoo)\ot \behh(a\mi)$.
\end{enumerate}

A bilinear map $Q: H\to A\ot A$ is called a cycle on $(H,\alhh,\behh)$ if
\begin{enumerate}
\item[(CC4)]  $ \alaa(x\qi)\ot \Delta_A(x\qii)+\alaa(x\boi)\ot Q(x\boo)=\Delta_A(x\qi)\ot \beaa(x\qii)+Q(x\boo)\ot \beaa(x\bi)$.
\end{enumerate}

In the following definitions, we introduced the concept of cocycle
   BiHom-algebras and cycle    BiHom-coalgebras, which are  in fact not really
ordinary    BiHom-algebras and    BiHom-coalgebras, but generalized ones.

\begin{definition}
(i): Let $\si$ be a cocycle on a vector space  $(H,\alhh,\behh)$ equipped with a multiplication $H \ot H \to H$, satisfying the the
following cocycle BiHom-associative identity:
\begin{enumerate}
\item[(CC5)] $\alhh(x)(y z)+\alhh(x) \triangleleft \sigma(y, z)=(x y) \behh(z)+\sigma(x, y) \trr \behh(z)$.
\end{enumerate}
Then  $(H,\alhh,\behh)$ is called a  $\si$-BiHom-algebra.

(ii): Let $\theta$ be a cocycle on a vector space $(A,\alaa,\beaa)$  equipped with a multiplication $A \ot A \to A$, satisfying the the
following cocycle BiHom-associative identity:
\begin{enumerate}
\item[(CC6)] $\alaa(a)(b c)+\alaa(a)\ppl \theta(b, c)=(a b) \beaa(c)+\theta(a b) \ppr \beaa(c)$.
\end{enumerate}
Then  $(A,\alaa,\beaa)$ is called a   $\theta$-BiHom-algebra.

(iii) Let $P$ be a cycle on a vector space  $(H,\alhh,\behh)$ equipped with a comultiplication $\Delta: H \to H \ot H$, satisfying the the
following cycle BiHom-coassociative identity:
\begin{enumerate}
\item[(CC7)] $\alhh(x_1)\ot \Delta_H(x_2)+\alhh(x\boo)\ot P(x\bi)=\Delta_H(x\li)\ot \behh(x\lii)+ P(x\boi) \ot \behh(x\boo)$,
\end{enumerate}
\noindent Then  $(H,\alhh,\behh)$ is called a  $P$-BiHom-coalgebra.

(iv) Let $Q$ be a cycle on a vector space  $(A,\alaa,\beaa)$ equipped with an commutativity  map $\Delta: A \to A \ot A$, satisfying the the
following cycle BiHom-coassociative identity:
\begin{enumerate}
\item[(CC8)] $\alaa(a\li)\ot \Delta_A(a\lii)+\alaa(a\mo)\ot Q(a\mi)=\Delta_A(a\li)\ot \beaa(a\lii)+Q(a\lmoi)\ot \beaa(a\lmoo)$,
\end{enumerate}
\noindent Then  $(A,\alaa,\beaa)$ is called a  $Q$-BiHom-coalgebra.
\end{definition}

\begin{definition}
A  \emph{cocycle cross product system } is a pair of  $\theta$-BiHom-algebra $(A,\alaa,\beaa)$ and $\sigma$-BiHom-algebra $(H,\alhh,\behh)$,
where $\si: H\ot H\to A$ is a cocycle on $H$, $\theta: A\ot A\to H$ is a cocycle on $A$ and the following conditions are satisfied:
\begin{enumerate}
\item[(CP1)] $(ab) \trr \behh(x)+\theta(a, b) \behh(x)=\alaa(a)\trr (b\trr x)+\theta(\alaa(a), b\ppl x)$,
\item[(CP2)] $\alhh(x) \trl (a b)+\alhh(x)\theta(a, b)=(x \trl a) \trl \beaa(b)+\theta(x \ppr a, \beaa(b))$,
\item[(CP3)] $(a \trr x) \trl \beaa(b)+\theta(a\ppl x, \beaa(b))=\alaa(a) \trr(x \trl b)+\theta(\alaa(a), x\ppl b)$,
\item[(CP4)] $ \alaa(a) \trr (x y)+\theta(\alaa(a), \sigma(x, y))=(a \trr x) \behh(y)+(a\ppl x) \trr \behh(y)$,
\item[(CP5)] $(x y) \trl \beaa(a)+\theta(\sigma(x, y), \beaa(a))=\alhh(x) \trl(y \ppr a)+\alhh(x)(y \trl a)$,
\item[(CP6)] $(x \triangleleft a) \behh(y)+(x \ppr a) \trr \behh(y)=\alhh(x)(a \trr y)+\alhh(x) \trl(a\ppl y)$,
\item[(CP7)] $(x y) \ppr \beaa(a)+\sigma(x, y) \beaa(a)=\alhh(x)\ppr(y\ppr a)+\sigma(\alhh(x), y\trl a)$,
\item[(CP8)] $\alaa(a)\ppl (x y)+\alaa(a) \sigma(x, y)=(a\ppl x)\ppl \behh(y)+\sigma(a \trr x, \behh(y))$,
\item[(CP9)] $(x \ppr a)\ppl \behh(y)+\sigma(x \trr a, \behh(y))=\alhh(x) \ppr(a\ppl y)+\sigma(\alhh(x), a\trr y)$,
\item[(CP10)] $\alhh(x) \ppr(a b)+\sigma(\alhh(x), \theta(a, b))=(x\ppr a) \beaa(b)+(x\trl a) \ppr \beaa(b)$,
\item[(CP11)] $(ab)\ppl \behh(x)+\sigma(\theta(a, b), \behh(x))=\alaa(a)(b\ppl x)+\alaa(a)\ppl (b \trr x)$,
\item[(CP12)] $(a\ppl x) \beaa(b)+(a\trr x) \ppr \beaa(b)=\alaa(a)(x \ppr b)+\alaa(a)\ppl (x\trl b)$.
\end{enumerate}
\end{definition}

\begin{lemma}
Let $(A, H)$ be  a  cocycle cross product system.
If we define $D=A_\sigma\#_\theta H$ as the vector space $A\oplus H$ with the   multiplication
\begin{align*}
&{\alpha_E}(a, x):=(\alpha_A(a), \alpha_H(x)),\quad{\beta_E}(a, x):=(\beta_A(a), \beta_H(x)),\\
&(a, x)(b, y)=\big(ab+x\ppr b+a\ppl y+\sigma(x, y), \, xy+x\trl b+a\trr y+\theta(a, b)\big),
\end{align*}
then  $A_\sigma\#_\theta H$ forms an BiHom-algebra which is called the cocycle cross product BiHom-algebra.
\end{lemma}

\begin{proof} We have to check $((a, x)(b, y)){\beta_E}(c, z)={\alpha_E}(a, x)((b, y)(c, z))$. By direct computations, the left hand side is equal to
\begin{eqnarray*}
&&{\alpha_E}(a, x)\big((b, y)(c, z)\big)\\
&=&\left(\alaa(a), \alhh(x)\right)\big(b c+y\ppr c+b\ppl z+\sigma(y, z), y z+y \trl c+b \trr z+\theta(b, c)\big)\\
&=&\Big(\alaa(a)(b c)+\alaa(a)(y \ppr c)+\alaa(a)(b\ppl z)+\alaa(a)(\sigma(y, z))+\alhh(x) \ppr(b c)\\
&&+\alhh(x) \ppr(y\ppr c)+\alhh(x) \ppr(b\ppl z)+\alhh(x) \ppr \sigma(y, z)\\
&&+\alaa(a)\ppl (y z)+\alaa(a)\ppl (y \trl c)+\alaa(a)\ppl (b \trr z)+\alaa(a)\ppl \theta(b, c)\\
&&+\sigma(\alhh(x), y z)+\sigma(\alhh(x), y \trl c)+\sigma(\alhh(x), b \trr z)+\sigma(\alhh(x), \theta(b, c)),\\
&&\quad \alhh(x)(y z)+\alhh(x)(y \trr c)+\alhh(x)(b \trr z)+\alhh(x) \theta(b, c)+\alhh(x) \trl (b c)\\
&&+\alhh(x) \trl(y \ppr c)+\alhh(x)\trl (b\ppl z)+\alhh(x) \trl \sigma(y, z)\\
&&+\alaa(a) \trr(y z)+\alaa(a) \trr(y\trl c)+\alaa(a) \trr(b \trr z)+\alaa(a) \trr \theta(b, c)\\
&&+\theta(\alaa(a), b c)+\theta(\alaa(a), y \ppr c)+\theta(\alaa(a), b\ppl z)+\theta(\alaa(a), \sigma(y, z))\Big)
\end{eqnarray*}
and the right hand side is equal to
\begin{eqnarray*}
&&\big((a, x)(b, y)\big){\beta_E}(c, z)\\
&=&\big(a b+x \ppr b+a\ppl y+\sigma(x, y), x y+x \trl b+a\trr y+\theta(a, b)\big)(\beaa(c), \behh(z))\\
&=&\Big( a b) \beaa(c)+(x\ppr b) \beaa(c)+(a\ppl y) \beaa(c)+\sigma(x, y) \beaa(c)+(x y) \ppr \beaa(c) \\
&&+(x \trl b) \ppr \beaa(c)+(a \trr y)\ppr \beaa(c)+\theta(a, b) \ppr \beaa(c)\\
&&+(a b) \ppl \behh(z)+(x \ppr b)\ppl \behh(z)+(a\ppl y)\ppl \behh(z)+\sigma(x y)\ppl \behh(z)\\
&&+\sigma(x y, \behh(z))+\sigma(x \trl b, \behh(z))+\sigma(a \trr y, \behh(z)) +\sigma(\theta(a, b), \behh(z)),\\
&&\quad (x y) \behh(z)+(x\trl b) \behh(z)+(a\trr y) \behh(z)+\theta(a, b) \behh(z)+(x y) \trl \beaa(c)\\
&&+(x \trl b) \trl \beaa(c)+(a \trr y) \trl \beaa(c)+\theta(a, b) \trl\beaa(c) \\
&&+(a b) \trr \behh(z)+(x \ppr b) \trr \behh(z)+(a\ppl y)\trr \behh(z)+\sigma(x, y) \trr \behh(z)\\
&&+\theta(a b, \beaa(c))+\theta(x \ppr b, \beaa(c))+\theta(a\ppl y, \beaa(c))+\theta(\sigma(x, y), \beaa(c))\Big).
\end{eqnarray*}
Thus the two sides are equal to each other if and only if (CP1)--(CP12) hold.
\end{proof}

\begin{definition}
A  \emph{cycle cross coproduct system } is a pair of   $P$-BiHom-coalgebra $(A,\alaa,\beaa)$ and  $Q$-BiHom-coalgebra $(H,\alhh,\behh)$ ,  where $P: A\to H\ot H$ is a cycle on $A$,  $Q: H\to A\ot A$ is a cycle over $(H,\alhh,\behh)$ such that following conditions are satisfied:
\begin{enumerate}
\item[(CCP1)] $\alhh(a_{(-1)}) \otimes \Delta_{A}\left(a_{(0)}\right)+\alhh(a\ppi) \otimes Q\left(a\pii\right)=\phi\left(a_{1}\right) \otimes \beaa(a_{2})+\gamma\left(a_{(-1)}\right) \otimes \beaa(a_{(0)})$,
\item[(CCP2)] $\Delta_{A}\left(a_{(0)}\right) \otimes \behh(a_{(1)})+Q\left(a\ppi\right) \otimes \behh(a\pii)=\alaa(a_{1}) \otimes \psi\left(a_{2}\right)+\alaa(a_{(0)}) \otimes \rho\left(a_{(1)}\right)$,
\item[(CCP3)] $\alaa(x_{[-1]}) \otimes \Delta_{H}\left(x_{[0]}\right)+\alaa(x\qi) \otimes P\left(x\qii\right)=\rho\left(x_{1}\right) \otimes \behh(x_{2})+\psi\left(x_{[-1]}\right) \otimes \behh(x_{[0]})$,
\item[(CCP4)] $\Delta_{H}(x\boo)\ot \beaa(x\bi)+P(x\qi)\ot \beaa(x\qii)=\alhh(x\boo)\ot \phi(x\bi)+\alhh(x_1)\ot \gamma(x_2)$,
\item[(CCP5)] $\alaa(a_{1}) \otimes \phi\left(a_{2}\right)+\alaa(a_{(0)}) \otimes \gamma\left(a_{(1)}\right)=\psi\left(a_{1}\right) \otimes \beaa(a_{2})+\rho\left(a_{(-1)}\right) \otimes \beaa(a_{(0)})$,
\item[(CCP6)] $\alhh(x_{1}) \otimes \rho\left(x_{2}\right)+\alhh(x_{[0]}) \otimes \psi\left(x_{[1]}\right)=\phi\left(x_{[-1]}\right) \otimes \behh(x_{[0]})+\gamma\left(x_{1}\right) \otimes \behh(x_{2})$,
\item[(CCP7)] $\alhh(a_{(-1)}) \otimes \phi\left(a_{(0)}\right)+\alhh(a\ppi)\otimes \gamma\left(a\pii\right)=\Delta_{H}\left(a_{(-1)}\right) \otimes \beaa(a_{(0)})+P\left(a_{1}\right) \otimes \beaa(a_{2})$,
\item[(CCP8)] $\alaa(a_{(0)})\ot \Delta_{H}\left(a_{(1)}\right)+\alaa(a_{1})\otimes P\left(a_{2}\right)=\psi\left(a_{(0)}\right) \otimes \behh(a_{(1)})+\rho\left(a\ppi\right) \otimes \behh(a\pii)$,
\item[(CCP9)] $\alhh(a_{(-1)}) \otimes \psi\left(a_{(0)}\right)+\alhh(a\ppi) \otimes \rho\left(a\pii\right)
=\phi\left(a_{(0)}\right) \otimes \behh(a_{(1)})+\gamma\left(a\ppi\right) \otimes \behh(a\pii)$,
\item[(CCP10)] $\alaa(x_{[-1]}) \otimes \rho\left(x_{[0]}\right)+\alaa(x\qi) \otimes \psi\left(x\qii\right)=\Delta_{A}\left(x_{[-1]}\right) \otimes \behh(x_{[0]})+Q\left(x_{1}\right) \otimes \behh(x_{2})$,
\item[(CCP11)] $\alhh(x\boo)\ot\Delta_A(x\bi)+\alhh(x_1)\ot Q(x_2)=\gamma(x\boo)\ot \beaa(x\bi)+\phi(x\qi)\ot \beaa(x\qii)$,
\item[(CCP12)] $\alaa(x\boi)\ot\gamma(x\boo)+\alaa(x\qi)\ot\phi(x\qii)=\rho(x\boo)\ot \beaa(x\bi)+\psi(x\qi)\ot \beaa(x\qii)$.
\end{enumerate}
\end{definition}

\begin{lemma}\label{lem2} Let $(A, H)$ be  a  cycle cross coproduct system. If we define $D=A^{P}\# {}^{Q} H$ as the vector
space $A\oplus H$ with the   comultiplication
$$\Delta_{E}(a,x)=(\Delta_{A}+\phi+\psi+P)(a)+(\Delta_{H}+\psi+\gamma+Q)(x), $$
that is
$$\Delta_{E}(a)= a\li \ot a\lii+ a\moi \ot a\mo+a\mo\ot a\mi+a\ppi\ot a\pii,$$
$$\Delta_{E}(x)= x\li \ot x\lii+ x\boi \ot x\boo+x\boo \ot x\bi+x\qi\ot x\qii,$$
then  $A^{P}\# {}^{Q} H$ forms a BiHom-coalgebra which we will call it the cycle cross coproduct   BiHom-coalgebra.
\end{lemma}

\begin{proof} We have to check $ ({\alpha_E}\otimes \Delta_E ) \Delta_E(a, x)= (\Delta_E \otimes {\beta_E}) \Delta_E(a, x)$. By direct computations, the left hand side is equal to
\begin{eqnarray*}
&& ({\alpha_E}\otimes \Delta_E ) \Delta_E(a, x)\\
&=&\alaa(a_{1}) \otimes \Delta_{A}\left(a_{2}\right)+\alaa(a_{1}) \otimes \phi\left(a_{2}\right)+\alaa(a_{1}) \otimes \psi\left(a_{2}\right)\\
&&+\alaa(a_{1}) \otimes P\left(a_{2}\right) +\alhh(a_{(-1)}) \otimes \Delta_{A}\left(a_{(0)}\right)+\alhh(a_{(-1)}) \otimes \phi\left(a_{(0)}\right)\\
&&+\alhh(a_{(-1)}) \otimes \psi\left(a_{(0)}\right)+\alhh(a_{(-1)}) \otimes P\left(a_{(0)}\right)+\alaa(a_{(0)}) \otimes \Delta_{H}\left(a_{(1)}\right)\\
&&+\alaa(a_{(0)}) \otimes \rho\left(a_{(1)}\right)+\alaa(a_{(0)}) \otimes \gamma\left(a_{(1)}\right)+\alaa(a_{(0)}) \otimes Q\left(a_{(1)}\right)\\
&&+\alhh(a\ppi) \otimes \Delta_{H}\left(a\pii\right)+\alhh(a\ppi) \otimes \rho\left(a\pii\right)\\
&&+\alhh(a\ppi) \otimes \gamma\left(a\pii\right)+\alhh(a\ppi)\otimes Q\left(a\pii\right)\\
&&+\alhh(x_{1}) \otimes \Delta_{H}\left(x_{2}\right)+\alhh(x_{1}) \otimes \rho\left(x_{2}\right)+\alhh(x_{1}) \otimes \gamma\left(x_{2}\right)\\
&&+\alhh(x_{1}) \otimes Q\left(x_{2}\right)+\alaa(x\boi) \otimes \Delta_{H}\left(x\boo\right)+\alaa(x\boi)\otimes \rho\left(x\boo\right)\\
&&+\alaa(x\boi) \otimes \gamma\left(x\boo\right)+\alaa(x\boi)\otimes Q\left(x\boo\right)+\alhh(x\boo) \otimes \Delta_{A}\left(x\bi\right)\\
&&+\alhh(x\boo) \otimes \phi\left(x\bi\right)+\alhh(x\boo)\otimes \psi\left(x\bi\right)+\alhh(x\boo) \otimes P\left(x\bi\right)\\
&&+\alaa(x\qi) \otimes \Delta_{A}\left(x\qii\right)+\alaa(x\qi) \otimes \phi\left(x\qii\right)\\
&&+\alaa(x\qi) \otimes \psi\left(x\qii\right)+\alaa(x\qi)\otimes P\left(x\qii\right)
\end{eqnarray*}
and the right hand side is equal to
\begin{eqnarray*}
&& (\Delta \otimes {\beta_E}) \Delta_E(a, x)\\
&=&\Delta_{A}\left(a_{1}\right) \otimes \beaa(a_{2})+\phi\left(a_{1}\right) \otimes \beaa(a_{2})+\psi\left(a_{1}\right) \otimes \beaa(a_{2})\\
&&+P\left(a_{1}\right) \otimes \beaa(a_{2})+\Delta_{H}\left(a_{(-1)}\right) \otimes \beaa(a_{(0)})+\rho\left(a_{(-1)}\right) \otimes \beaa(a_{(0)})\\
&&+\gamma\left(a_{(-1)}\right) \otimes \beaa(a_{(0)})+Q\left(a_{(-1)}\right) \otimes \beaa(a_{(0)})\\
&&+\Delta_{A}\left(a_{(0)}\right) \otimes \behh(a_{(1)})+\phi\left(a_{(0)}\right) \otimes \behh(a_{(1)})\\
&&+\psi\left(a_{(0)}\right) \otimes \behh(a_{(1)})+P\left(a_{(0)}\right) \otimes \behh(a_{(1)})\\
&&+\Delta_{H}\left(a\ppi\right) \otimes \behh(a\pii)+\rho\left(a\ppi\right) \otimes \behh(a\pii)\\
&&+\gamma\left(a\ppi\right) \otimes \behh(a\pii)+Q\left(a\ppi\right) \otimes \behh(a\pii)\\
&&+\Delta_{H}\left(x_{1}\right) \ot \behh(x_{2})+\rho\left(x_{1}\right) \otimes \behh(x_{2})\\
&&+\gamma\left(x_{1}\right) \otimes \behh(x_{2})+Q\left(x_{1}\right) \otimes \behh(x_{2})\\
&&+\Delta_{A}\left(x\boi\right) \otimes \behh(x\boo)+ \phi\left(x\boi\right) \otimes \behh(x\boo)\\
&&+\psi\left(x\boi\right)\otimes \behh(x\boo)+P\left(x\boi\right)\otimes \behh(x\boo)\\
&&+\Delta_{H}\left(x\boo\right) \otimes \beaa(x\bi)+ \rho\left(x\boo\right) \otimes \beaa(x\bi)\\
&&+\gamma\left(x\boo\right)\otimes \beaa(x\bi)+Q\left(x\boo\right)\otimes \beaa(x\bi)\\
&&+\Delta_{A}\left(x\qi\right) \otimes \beaa(x\qii)+\phi\left(x\qi\right) \otimes \beaa(x\qii)\\
&&+\psi\left(x\qi\right) \otimes \beaa(x\qii)+P\left(x\qi\right) \otimes \beaa(x\qii).
\end{eqnarray*}
Thus the two sides are equal to each other if and only if (CCP1)--(CCP12) hold.
\end{proof}

\begin{definition}
Let $A, H$ be both   BiHom-algebras and   BiHom-coalgebras. If  the following conditions hold:
\begin{enumerate}
\item[(CDM1)]  $\phi(a b)+\gamma(\theta(a, b))$ \\
$=\alhh(a_{(-1)}) \otimes\left(a_{(0)} \beaa(b)\right)+\alhh(a\ppi) \otimes (a\pii \ppr \beaa(b))\\
    =(\alaa(a) \trr b_{(-1)}) \otimes \beaa(b_{(0)})+\theta\left(\alaa(a), b_{1}\right) \otimes \beaa(b_{2})$,
\item[(CDM2)] $\psi(a b)+\rho(\theta(a, b))$ \\
$=(\alaa(a) b_{(0)}) \otimes \behh(b_{(1)})+(\alaa(a)\ppl b\ppi) \otimes \behh(b\pii)\\
  =\alaa(a_{(0)}) \otimes\left(a_{(1)} \trl \beaa(b)\right)+\alaa(a_{1}) \otimes \theta\left(a_{2},\beaa(b)\right)$,
\item[(CDM3)] $\rho(x y)+\psi(\sigma(x, y))$ \\
$=\alaa(x_{[-1]}) \otimes\left(x_{[0]} \behh(y)\right) +\alaa(x\qi)\otimes x\qii\trr \behh(y)\\
   =\left(\alhh(x) \ppr y_{[-1]}\right) \otimes \behh(y_{[0]})+\sigma\left(\alhh(x), y_{1}\right) \otimes \behh(y_{2})$,
\item[(CDM4)] $\gamma(x y)+\phi(\sigma(x, y))$ \\
$=\alhh(x_{[0]})\otimes (x_{[1]}\ppl \behh(y))+\alhh(x_{1}) \otimes \sigma\left(x_{2}, \behh(y)\right)\\
    =\alhh(x)y_{[0]}\otimes \beaa(y_{[1]})+\alhh(x)\trl y\qi\otimes \beaa(y\qii)$,
\item[(CDM5)] $\Delta_{A}(x \ppr b)+Q(x\trl  b)$ \\
$=\alaa(x_{[-1]}) \otimes\left(x_{[0]} \ppr \beaa(b)\right)+\alaa(x\qi) \otimes x\qii \beaa(b)\\
=\left(\alhh(x) \ppr b_{1}\right) \otimes \beaa(b_{2})+\sigma(\alhh(x), b_{(-1)}) \otimes \beaa(b_{(0)})$,
\item[(CDM6)] $\Delta_{A}(a\ppl y)+Q(a\trr y)$\\
$=\alaa(a_{1}) \otimes\left(a_{2} \ppl \behh(y)\right)+\alaa(a_{(0)}) \otimes \sigma\left(a_{(1)}, \behh(y)\right)\\
=\left(\alaa(a)\ppl y_{[0]}\right) \otimes \beaa(y_{[1]})+\alaa(a) y\qi \otimes \beaa(y\qii)$,
\item[(CDM7)] $\Delta_{H}(a \trr y)+P(a\ppl y)$\\
$=\alhh(a_{(-1)}) \otimes\left(a_{(0)}\trr \behh(y)\right)+\alhh(a\ppi) \otimes a\pii \behh(y)\\
=\left(\alaa(a) \trr y_{1}\right) \otimes \behh(y_{2})+\theta\left(\alaa(a), y_{[-1]}\right) \otimes \behh(y_{[0]})$,
\item[(CDM8)] $\Delta_{H}(x \trl b)+P(x\ppr b)$\\
$=\alhh(x_{1}) \otimes\left(x_{2}  \trl \beaa(b)\right)+\alhh(x_{[0]}) \otimes \theta\left(x_{[1]}, \beaa(b)\right)\\
=\left(\alhh(x)\trl b_{(0)}\right) \otimes \behh(b_{(1)})+\alhh(x) b\ppi\otimes \behh(b\pii)$,

\item[(CDM9)]$\Delta_{H}(\theta(a,b))+P(a, b)$\\
$=\alhh(a_{(-1)}) \otimes\theta(a_{(0)},\beaa(b))+a\ppi \otimes \alhh(a\pii)\trl \beaa(b)\\
=\theta(\alaa(a),b_{(0)})\otimes \behh(b_{(1)})+\alaa(a)\trr b\ppi\otimes \behh(b\pii)$,

\item[(CDM10)]$\Delta_{A}(\sigma(x,y))+Q(x, y)$\\
$=\alhh(x_{[-1]})\otimes \sigma(x_{[0]},\behh(y))+\alaa(x\qi)\otimes x\qii\ppl \behh(y)\\
=\sigma(\alhh(x),y_{[0]})\otimes \beaa(y_{[-1]})+\alhh(x)\ppr y\qi\otimes \beaa(y\qii)$,

\item[(CDM11)]
 $\phi(x \ppr b)+\gamma(x\trl b)$\\
 $=\alhh(x_{1}) \otimes\left(x_{2} \ppr\beaa( b)\right)+\alhh(x_{[0]}) \otimes x_{[1]} \beaa(b)\\
  =\alhh(x) b_{(-1)} \otimes \beaa(b_{(0)})+\left(\alhh(x)\trl b_{1}\right) \otimes \beaa(b_{2})$,

\item[(CDM12)]
$\psi(a\ppl y)+\rho(a \trr y)$\\
$=\alaa(a_{(0)}) \otimes a_{(1)} \behh(y)+\alaa(a_{1}) \otimes\left(a_{2} \trr \behh(y)\right)$\\
$=\left(\alaa(a)\ppl y_{1}\right) \otimes \behh(y_{2})+\alaa(a) y_{[-1]} \otimes \behh(y_{[0]})$,

\item[(CDM13)]
$\psi(x \ppr b)+\rho(x\trl b)$\\
$=(\alhh(x) \ppr b_{(0)}) \otimes \behh(b_{(1)})+\alaa(x\qi)\otimes\theta(x\qii,\beaa(b))$\\
$=\alaa(x_{[-1]}) \otimes\left(x_{[0]} \trl \beaa(b)\right)+\sigma(\alhh(x),b_{<1>})\otimes\behh(b_{<2>})$,

\item[(CDM14)]
$\phi(a\ppl y)+\gamma(a \trr y)$\\
$=\alhh(a_{(-1)}) \otimes\left(a_{(0)}\ppl \behh(y)\right)+\alhh(a_{<1>})\otimes\sigma(a_{<2>,\behh(y)})$\\
$=\left(\alaa(a)\trr y_{[0]}\right)\otimes \beaa(y_{[1]})+\theta(\alaa(a),y\qi)\otimes \beaa(y\qii)$.
\end{enumerate}
\noindent then $(A, H)$ is called a \emph{cocycle double matched pair}.
\end{definition}

\begin{definition}\label{bi-cycle1}
(i) A \emph{cocycle braided BiHom-Frobenius algebra} $(H,\alhh,\behh)$ is simultaneously a cocycle BiHom-algebra and a cycle BiHom-coalgebra satisfying the condition
\begin{eqnarray}
\notag\Delta_{H}(x y)+P\sigma(x,y)&=&\alhh(x_{1}) \otimes x_{2} \behh(y)+\alhh(x_{[0]}) \otimes\left(x_{[1]} \trr \behh(y)\right)\\
&=&\alhh(x) y_{1} \otimes \behh(y_{2})+\left(\alhh(x) \trl y_{[-1]}\right) \otimes \behh(y_{[0]}).
\end{eqnarray}
(ii) A \emph{cocycle braided BiHom-Frobenius algebra} $(A,\alaa,\beaa)$ is simultaneously a cocycle  BiHom-algebra and  a cycle BiHom-coalgebra satisfying the  condition
\begin{eqnarray}
\notag\Delta_{A}(a b)+Q\theta(a,b)&=&\alaa(a_{1}) \otimes a_{2} \beaa(b)+\alaa(a_{(0)}) \otimes\left(a_{(1)} \ppr  \beaa(b)\right)\\
&=&\alaa(a) b_{1} \otimes \beaa(b_{2})+(\alaa(a)\ppl b_{(-1)}) \otimes \beaa(b_{(0)}).
\end{eqnarray}
\end{definition}

The next theorem means that we can obtain a BiHom-Frobenius algebra from two cocycle braided BiHom-Frobenius algebras.
\begin{theorem}\label{main2}
Let $(A, H)$ be a cocycle cross product system and a cycle cross
coproduct system. Then the cocycle cross product BiHom-algebra and
cycle cross coproduct  BiHom-coalgebra fit together to form an ordinary
BiHom-Frobenius algebra if and only if $(A,\alaa, \beaa)$, $(H,\alhh,\behh)$ are cocycle braided BiHom-Frobenius algebras  and  the conditions (CDM1)--(CDM14) are satisfied. We will call it the cocycle bicrossproduct BiHom-Frobenius algebra and denote it by $A^{P}_{\sigma}\# {}^{Q}_{\theta}H$.
\end{theorem}

\begin{proof}  We only need to check the compatibility condition
$$\Delta_E((a, x) (b, y))=({\alpha_E} \otimes \mu_E)(\Delta_E(a, x) \cdot{\beta_E}(b, y))=(\mu_E \otimes {\beta_E})({\alpha_E}(a, x) \cdot \Delta(b, y)).$$
The left hand side is equal to
\begin{eqnarray*}
&&\Delta_E((a, x) (b, y))\\
&=&\Delta_E(ab+x \ppr b+a\ppl y+\sigma(x, y), x y+x \trl b+a\trr y+\theta(a, b))\\
&=&\Delta_A(a b)+\phi(a b)+\psi(a b)+\Delta_A(x \ppr b)+\phi(x \ppr b)+\psi(x \ppr b)\\
&&+\Delta_{A}(a\ppl y)+\phi(a\ppl y)+\psi(a\ppl y)+\Delta_{H}(x y)+\rho(x y)+\gamma(x y)\\
&&+\Delta_{H}(x \trl b)+\rho(x \trl b)+\gamma(x \trl b)+\Delta_{H}(a \trr y)+\rho(a \trr y)+\gamma(a \trr y),
\end{eqnarray*}
the middle and right hand sides are equal to
\begin{eqnarray*}
&&({\alpha_E} \otimes \mu_E)(\Delta_E(a, x) \cdot{\beta_E}(b, y))\\
&=&\alaa(a_{1}) \otimes a_{2} \beaa(b)+\alaa(a_{1}) \otimes\left(a_{2}\ppl \behh(y)\right)+\alaa(a_{1}) \otimes\left(a_{2} \trr \behh(y)\right)\\
&&+\alaa(a_{1}) \otimes \theta(a_2, \beaa(b))+\alhh(a_{(-1)}) \otimes a_{(0)} \beaa(b)+\alhh(a_{(-1)}) \otimes\left(a_{(0)}\ppl \behh(y)\right)\\
&&+\alhh(a_{(-1)}) \otimes\left(a_{(0)} \trr \behh(y)\right)+\alhh(a_{(-1)}) \otimes \theta(a_{(0)}, \beaa(b))\\
&&+\alaa(a_{(0)})\otimes\left(a_{(1)} \ppr \beaa(b)\right)+\alaa(a_{(0)}) \otimes\left(a_{(1)} \trl \beaa(b)\right)\\
&&+\alaa(a_{(0)}) \otimes a_{(1)}\behh(y)+\alaa(a_{(0)}) \otimes \sigma(a_{(1)}, \behh(y))\\
&&+\alhh(a\ppi) \otimes \left(a\pii \ppr\beaa( b)\right)+\alhh(a\ppi) \ot \sigma\left(a\pii, \behh(y)\right)\\
&&+\alhh(a\ppi)\otimes a\pii\behh(y)+\alhh(a\ppi) \otimes \left(a\pii\trl  \beaa(b)\right)\\
&&+\alhh(x_{1}) \otimes\left(x_{2}\ppr \beaa(b)\right)+\alhh(x_{1}) \otimes\left(x_{2} \trl \beaa(b)\right)\\
&&+\alhh(x_{1}) \otimes x_{2} \behh(y) +\alhh(x_{1}) \otimes \sigma(x_{2}, \behh(y))+\alaa(x\boi)\ot(x\boo\ppr \beaa(b))\\
&&+\alaa(x\boi)\ot(x\boo\trl \beaa(b))+\alaa(x\boi)\ot x\boo \behh(y)+\alaa(x\boi)\ot \sigma(x\boo, \behh(y))\\
&&+\alhh(x\boo)\ot x\bi \beaa(b)+\alhh(x\boo)\ot (x\bi\ppl \behh(y))+\alhh(x\boo)\ot (x\bi\trr \behh(y))\\
&&+\alhh(x\boo)\ot \theta(x\bi, \beaa(b))+\alaa(x\qi)\otimes x\qii \beaa(b)+\alaa(x\qi)\otimes \left(x\qii \ppl \behh(y)\right)\\
&&+\alaa(x\qi)\otimes \left(x\qii \trr \behh(y)\right)+\alaa(x\qi) \otimes \theta\left(x\qii, \beaa(b)\right),
\end{eqnarray*}
\begin{eqnarray*}
&&(\mu_E \otimes {\beta_E})({\alpha_E}(a, x) \cdot \Delta(b, y))\\
&=&\alaa(a) b_{1} \otimes \beaa(b_{2})+\left(\alhh(x) \ppr b_{1}\right) \otimes \beaa(b_{2})+\left(\alhh(x) \trl b_{1}\right) \otimes \beaa(b_{2})\\
&&+\theta(\alaa(a), b_{1})\otimes \beaa(b_{2})+(\alaa(a)\ppl b_{(-1)}) \otimes \beaa(b_{(0)})+(\alaa(a) \trr b_{(-1)}) \otimes \beaa(b_{(0)})\\
&&+\alhh(x) b_{(-1)} \otimes \beaa(b_{(0)})+\sigma(\alhh(x), b_{(-1)})\otimes \beaa(b_{(0)})+\alaa(a) b_{(0)} \otimes \behh(b_{(1)})\\
&&+(\alhh(x) \ppr b_{(0)}) \otimes \behh(b_{(1)})+(\alhh(x) \trl b_{(0)}) \otimes \behh(b_{(1)})\\
&&+\theta(\alaa(a), b_{(0)}) \otimes \behh(b_{(1)})+(\alaa(a)\ppl b\ppi) \otimes \behh(b\pii)\\
&&+(\alaa(a) \trr b\ppi) \otimes \behh(b\pii)+\alhh(x) b\ppi \otimes \behh(b\pii)\\
   &&+\sigma(\alhh(x), b\ppi)\otimes \behh(b\pii)+\left(\alaa(a)\ppl y_{1}\right) \otimes \behh(y_{2})+\left(\alaa(a) \trr y_{1}\right) \otimes \behh(y_{2})\\
   &&+\alhh(x) y_{1}\ot \behh(y_{2})+\sigma(\alhh(x), y_{1})\ot \behh(y_{2})+\alaa(a)y\boi\ot \behh(y\boo)\\
   &&+(\alhh(x)\ppr y\boi)\ot \behh(y\boo)+(\alhh(x)\trl y\boi)\ot \behh(y\boo)\\
   &&+\theta(\alaa(a),y\boi)\ot \behh(y\boo)+(\alaa(a)\ppl y\boo)\ot \beaa(y\bi)\\
   &&+(\alaa(a)\trr y\boo)\ot \beaa(y\bi)+\alhh(x)y\boo\ot \beaa(y\bi)\\
&&+\sigma(\alhh(x), y\boo)\ot \beaa(y\bi)+(\alaa(a)y\qi)\ot \beaa(y\qii)+(\alhh(x)\ppr y\qi)\ot \beaa(y\qii)\\
&&+(\alhh(x)\trl y\qi)\ot \beaa(y\qii)+\theta(\alaa(a), y\qi)\ot \beaa(y\qii).
\end{eqnarray*}
Thus both sides are equal to each other if and only if $(A,\alaa, \beaa)$, $(H,\alhh,\behh)$ are cocycle braided BiHom-Frobenius algebras  and  the conditions (CDM1)--(CDM14) are satisfied.  The proof is completed.
\end{proof}

\section{Extending structures for BiHom-Frobenius algebras}
In this section, we will study the extending problem for  a BiHom-Frobenius algebra.

\begin{definition}
Let $({A},\cdot)$ be a given   BiHom-algebra (BiHom-coalgebra,   BiHom-bialgebra), $E$ a BiHom-vector space.
An extending system of ${A}$ through $V$ is a BiHom-algebra  (BiHom-coalgebra,   BiHom-bialgebra) on $E$
such that $V$ is a complement subspace of ${A}$ in $E$, the canonical injection map $i: A\to E, a\mapsto (a, 0)$  or the canonical projection map $p: E\to A, (a,x)\mapsto a$ is a  BiHom-algebra (BiHom-coalgebra,   BiHom-bialgebra) homomorphism.
The extending problem is to describe and classify up to an isomorphism  the set of all   BiHom-algebra  (BiHom-coalgebra,   BiHom-bialgebra) structures that can be defined on $E$.
\end{definition}

\begin{definition}
Let ${A} $ be a  BiHom-algebra (BiHom-coalgebra, BiHom-bialgebra), $E$  be a BiHom-algebra  (BiHom-coalgebra, BiHom-bialgebra) such that
${A} $ is a subspace of $E$ and $V$ a complement of
${A} $ in $E$. For a linear map $\varphi: E \to E$ we consider the diagram:
\begin{equation}\label{eq:ext1}
\xymatrix{
   0  \ar[r]^{} &A \ar[d]_{\id_A} \ar[r]^{i} & E \ar[d]_{\varphi} \ar[r]^{\pi} &V \ar[d]_{\id_V} \ar[r]^{} & 0 \\
   0 \ar[r]^{} & A \ar[r]^{i'} & {E} \ar[r]^{\pi'} & V \ar[r]^{} & 0.
   }
\end{equation}
where $\pi : E \to V$ is the canonical projection of $E ={A}  \oplus V$ onto $V$ and $i: {A}  \to E$ is the inclusion map. We say that $\varphi: E \to E$ \emph{stabilizes}
${A} $ if the left square of the diagram \eqref{eq:ext1} is  commutative.
Let $(E, \cdot)$ and $(E,\cdot')$ be two BiHom-algebra (BiHom-coalgebra,   BiHom-bialgebra) structures on $E$. $(E, \cdot)$ and $(E, \cdot')$ are called \emph{equivalent}, and we denote this by $(E, \cdot) \equiv (E, \cdot')$, if there exists a BiHom-algebra (BiHom-coalgebra,  BiHom-bialgebra) isomorphism $\varphi: (E, \cdot)
\to (E, \cdot')$ which stabilizes ${A} $. Denote by $Extd(E,{A} )$ ($CExtd(E,{A} )$, $BExtd(E,{A} )$) the set of equivalent classes of  BiHom-algebra (BiHom-coalgebra,  BiHom-bialgebra) structures on $E$.
\end{definition}

\subsection{Extending structures for BiHom-algebras and BiHom-coalgebras}
First we are going to study extending problem for BiHom-algebras and BiHom-coalgebras.

There are two cases for $A$ to be a BiHom-algebra in the cocycle cross product system defined in last section, see condition (CC6). The first case is when we let $\ppr$, $\ppl$ to be trivial and $\theta\neq 0$,  then from conditions (CP10) and (CP11) we get $\si(x, \theta(a, b))=\si(\theta(a, b), x)=0$, since $\theta\neq 0$ we assume $\sigma=0$ for simplicity, thus  we obtain the following type $(a1)$  unified product for algebras.

\begin{lemma}
Let ${A}$ be a BiHom-algebra and $V$ a BiHom-vector space. An extending datum of ${A}$ by $V$ of type (a1)  is  $\Omega^{(1)}({A},V)=(\trr,\trl, \theta)$ consisting of bilinear maps
\begin{eqnarray*}
\trr: A\otimes V \to V,\quad \trl: V\otimes A \to V,\quad\theta: A\ot A\to V, \quad\cdot_V: V\ot V\to V.
\end{eqnarray*}
Denote by $A_{\theta}\#_{}V$ the vector space $E={A}\oplus V$ together with the multiplication given by
\begin{eqnarray*}
&&{\alpha_E}(a, x):=(\alpha_A(a), \alpha_V(x)),\quad{\beta_E}(a, x):=(\beta_A(a), \beta_V(x)),\\
&&(a, x)(b, y)=\big(ab, \, xy+x\trl b+a\trr y+\theta(a, b)\big).
\end{eqnarray*}
Then $A_{}\# {}_{\theta}V$ is a BiHom-algebra if and only if the following compatibility conditions hold for all $a$, $b\in {A}$, $x$, $y$, $z\in V$:
\begin{enumerate}
\item[(A1)] $(ab) \trr \bevv(x)+\theta(a, b) \bevv(x)=\alaa(a)\trr (b\trr x)$,
\item[(A2)] $\alvv(x) \trl (a b)+\alvv(x)\theta(a, b)=(x \trl a) \trl \beaa(b)$,
\item[(A3)] $(a \trr x) \trl \beaa(b)=\alaa(a) \trr(x \trl b)$,
\item[(A4)] $ \alaa(a) \trr (x y)=(a \trr x) \bevv(y)$,
\item[(A5)] $(x y) \trl \beaa(a)=\alvv(x)(y \trl a)$,
\item[(A6)] $(x \trl a) \bevv(y)=\alaa(x)(a \trr y)$,
\item[(A7)] $\theta(a b, \beaa(c))+\theta(a, b) \triangleleft \beaa(c)=\alaa(a)\trr \theta(b, c)+\theta(\alaa(a), b c)$,
\item[(A8)] $(x y)\bevv(z)=\alvv(x)(yz)$,
\end{enumerate}
\end{lemma}
Note that (A1)--(A6)  are deduced from (CP1)--(CP6) and by (A8)  we obtain that $V$ is a BiHom-algebra. Furthermore, $V$ is in fact a subalgebra of $A_{}\#_{\theta}V$  but $A$ is not although $A$ is itself a BiHom-algebra.

Denote the set of all  algebraic extending datum of ${A}$ by $V$ of type (a1)  by $\mathcal{A}^{(1)}({A},V)$.


In the following, we always assume that $A$ is a subspace of a BiHom-vector space $E$, there exists a projection map $p: E \to{A}$ such that $p(a) = a$, for all $a \in {A}$.
Then the kernel space $V := \ker(p)$ is also a subspace of $E$ and a complement of ${A}$ in $E$.

\begin{lemma}\label{lem:33-1}
Let ${A}$ be a BiHom-algebra and $E$ a BiHom-vector space containing ${A}$ as a subspace.
Suppose that there is a BiHom-algebra structure on $E$ such that $V$ is a  subalgebra of $E$
and the canonical projection map $p: E\to A$ is a BiHom-algebra homomorphism.
Then there exists a BiHom-algebraic extending datum $\Omega^{(1)}({A},V)$ of ${A}$ by $V$ such that
$E\cong A_{}\#_{\theta}V$.
\end{lemma}

\begin{proof}
Since $V$ is a  subalgebra of $E$, we have $x\cdot_E y\in V$ for all $x, y\in V$.
We define the extending datum of ${A}$ through $V$ by the following
\begin{eqnarray*}
\trl: V\otimes {A} \to V, \qquad {x} \triangleleft {a} &:=&{x}\cdot_E {a},\\
\theta: A\otimes A \to V, \qquad \theta(a,b) &:=&a\cdot_E b-p \bigl(a\cdot_E b\bigl),\\
{\cdot_V}: V \otimes V \to V, \qquad {x}\cdot_V {y}&:=& {x}\cdot_E{y}
\end{eqnarray*}
for any $a , b\in {A}$ and $x, y\in V$. It is easy to see that the above maps are  well defined and
$\Omega^{(1)}({A}, V)$ is an extending system of
${A}$ trough $V$ and
\begin{eqnarray*}
\varphi:A_{}\#_{\theta}V\to E, \qquad \varphi(a, x) := a+x
\end{eqnarray*}
is an isomorphism of BiHom-algebras.
\end{proof}

\begin{lemma}\label{lem:33-100}
Let $\Omega^{(1)}({A},V)$ and ${\Omega}'^{(1)}({A},V)$ be two algebraic extending datums of ${A}$ by $V$ of type (a1) and $A_{\theta}\#_{\trr, \trl} V$, $A_{\theta'}\#_{\trr', \trl'} V$ be the corresponding unified products. Then there exists a bijection between the set of all homomorphisms of Lie algebras $\varphi:A_{\theta}\#_{\trr, \trl} V\to A_{\theta'}\#_{\trr', \trl'} V$ whose restriction on ${A}$ is the identity map and the set of pairs $(r,s)$, where $r:V\rightarrow {A}$ and $s:V\rightarrow V$ are two linear maps satisfying
\begin{enumerate}
\item[(A1)] ${r}(x\trl a)={r}(x)a,$
\item[(A2)]  $a\cdot' b=ab+r\theta(a,b),$
\item[(A3)]  ${r}(xy)={r}(x)\cdot' {r}(y),$
\item[(A4)]  ${s}(x)\trl' a+\theta'(r(x), a)={s}(x\trl a),$
\item[(A5)]  $\theta'(a,b)=s\theta(a,b),$
\item[(A6)]  ${s}(xy)={s}(x)\cdot' {s}(y)+{s}(x)\trl'{r}(y)-{s}(y)\trl'{r}(x)+\theta'(r(x), r(y)).$
\end{enumerate}
for all $a\in{A}$ and $x$, $y\in V$.

Under the above bijection the homomorphism of algebras $\varphi=\varphi_{r,s}: A_{}\#_{\theta}V\to A_{}\#_{\theta'} V$ to $(r,s)$ is given  by $\varphi(a,x)=(a+r(x), s(x))$ for all $a\in {A}$ and $x\in V$. Moreover, $\varphi=\varphi_{r,s}$ is an isomorphism if and only if $s: V\rightarrow V$ is a linear isomorphism.
\end{lemma}

\begin{proof}
Let $\varphi: A_{}\#_{\theta}V\to A_{}\#_{\theta'} V$  be an algebra homomorphism  whose restriction on ${A}$ is the identity map. Then $\varphi$ is determined by two linear maps $r: V\rightarrow {A}$ and $s: V\rightarrow V$ such that
$\varphi(a,x)=(a+r(x),s(x))$ for all $a\in {A}$ and $x\in V$.
In fact, we have to show
$$\varphi((a, x)(b, y))=\varphi(a, x)\cdot'\varphi(b, y).$$
The left hand side is equal to
\begin{eqnarray*}
&&\varphi((a, x)(b, y))\\
&=&\varphi\left({ab},\,  x\trl b-y\trl a+{xy}+\theta(a,b)\right)\\
&=&\big({ab}+ r(x\trl b)-r(y\trl a)+r({xy})+r\theta(a,b),\\
&&\qquad\quad s(x\trl b)-s(y\trl a)+s({xy})+s\theta(a,b)\big),
\end{eqnarray*}
and the right hand side is equal to
\begin{eqnarray*}
&&\varphi(a, x)\cdot' \varphi(b, y)\\
&=&(a+r(x),s(x))\cdot'  (b+r(y),s(y))\\
&=&\big((a+r(x))\cdot' (b+r(y)),  s(x)\trl'(b+r(y))-s(y)\trl'(a+r(x))\\
&&\qquad\qquad +s(x)\cdot' s(y)+\theta'(a+r(x),b+r(y))\big).
\end{eqnarray*}
Thus $\varphi$ is a homomorphism of BiHom-algebras if and only if the above conditions (A1)--(A6) hold.
\end{proof}

The second case is when $\theta=0$,  we obtain the following type (a2)  unified product for algebras.

\begin{lemma}
Let $A$ be a BiHom-algebra and $V$ a  BiHom-vector space. An \textit{extending
datum of $A$ through $V$} of type (a2)  is a system $\Omega(A, V) =
\bigl(\triangleleft, \, \triangleright, \, \leftharpoonup, \,
\rightharpoonup, \, \sigma, \, \cdot \bigl)$ consisting of linear maps
\begin{eqnarray*}
&&\ppr: V \otimes A \to A, \quad \ppl: A \otimes V \to A, \quad\triangleleft : V \otimes A \to V, \quad \trr: A \otimes V \to V \\
&& \sigma: V\otimes V \to A, \quad \cdot_V: V\otimes V \to V
\end{eqnarray*}
Denote by $A_{}\# {}_{\sigma}H$ the vector space $E={A}\oplus V$ together with the multiplication
\begin{align*}
{\alpha_E}(a, x):=(\alpha_A(a), \alpha_V(x)),\quad{\beta_E}(a, x):=(\beta_A(a), \beta_V(x)),\\
(a, x)(b, y)=\big(ab+x\ppr b+a\ppl y+\sigma(x, y), \, xy+x\trl b+a\trr y\big).
\end{align*}
Then $A_{}\# {}_{\sigma}H$  is a BiHom-algebra if and only if the following compatibility conditions hold for any $a, b, c\in A$,
$x, y, z\in V$:
\begin{enumerate}
\item[(B1)]   $( {a}     {b} ) \trr \bevv(x) =  \alaa({a} ) \trr ( {b}  \trr x),      \alvv(x) \trl( {a}{b} )  = (x \trl {a} ) \trl \beaa( {b} ),\\
                       ( {a} \trr x)\triangleleft  \beaa(b)   =  \alaa({a}) \trr (x \triangleleft  {b} )$,
\item[(B2)] $\alaa({a} ) \trr (x   y)= ( {a}  \ppl x) \trr \bevv(y) + ( {a}  \trr x)   \bevv(y)  $,
\item[(B3)] $ ( x   y) \trl  \beaa({a} )= \alvv( x) \trl (y\ppr  {a} ) + \alvv( x)    ( y\trl  {a} )$,
\item[(B4)] $ (x\trl  {a} )   \bevv(y) +(x\ppr  {a} ) \trr \bevv(y) = \alvv(x)    (  {a}  \trr y)+ \alvv(x) \trl ( {a}  \ppl y)$,
\item[(B5)] $(x y)\ppr  \beaa({a})+\sigma(x, y)  \beaa({a}) =  \alvv(x) \ppr (y\ppr  {a}  )+ \sigma( \alvv(x) ,y\trl  {a} )$,
\item[(B6)] $ \alaa({a})  \ppl (x   y)+ \alaa({a}) \sigma(x, y)= ( {a}  \ppl x) \ppl \bevv(y)+ \sigma( {a}  \trr x, \bevv(y))$,
\item[(B7)] $(x\ppr  {a} ) \ppl \bevv(y)+ \sigma(x\trl  {a} , \bevv(y))=\alvv( x )\ppr ( {a}  \ppl y) + \sigma( \alvv(x) ,  {a}  \trr y) $,
\item[(B8)] $ \alvv(x)\ppr (  {a}  {b} )=(x\ppr  {a} )  \beaa({b}) + (x\trl  {a} )\ppr  \beaa({b})$,
\item[(B9)] $(  {a}  {b} ) \ppl \bevv(x) = \alaa({a}) ( {b}  \ppl x) + \alaa( {a})  \ppl (  {b}  \trr x)$,
\item[(B10)] $( {a}  \ppl x) \beaa( {b}) + ( {a}  \trr x)\ppr \beaa( {b})  =  \alaa( {a}) (x\ppr  {b} )+ \alaa( {a})  \ppl (x\trl  {b} )$,
\item[(B11)] $ \sigma(x, y) \ppl  \bevv(z) + \sigma(x   y,  \bevv(z) )= \alvv(x) \ppr \sigma(y, z) + \sigma( \alvv(x) ,y   z) $,
\item[(B12)] $\sigma(x, y) \trr  \bevv(z) + (x   y)   \bevv( z)  =  \alvv(x) \trl \sigma(y, z)+  \alvv(x)  ( y   z)$.
\end{enumerate}
\end{lemma}

Similar as the proof of Lemma \ref{lem:33-1} and  Lemma \ref{lem:33-100} we obtain the following results.
\begin{lemma}\label{lem:33-2}
Let $A$ be a BiHom-algebra, $E$ a  vector space containing $A$ as a subspace.
If there is a BiHom-algebra structure on $E$ such that $A$ is a subalgebra of $E$. Then there exists a BiHom-algebra
extending structure $\Omega(A, V) = \bigl(\triangleleft, \, \triangleright, \,
\leftharpoonup, \, \rightharpoonup, \, \sigma, \, \cdot \bigl)$ of $A$ through $V$ such that there is an isomorphism of algebras $E\cong A_{\sigma}\#_{}H$.
\end{lemma}

\begin{lemma}\label{lem:33-200}
Let $\Omega(A, V) = \bigl(\triangleleft, \, \triangleright, \,
\leftharpoonup, \, \rightharpoonup, \, \sigma, \, \cdot \bigl)$ and
$\Omega(A, V) = \bigl(\triangleleft ', \, \triangleright ', \,
\leftharpoonup ', \, \rightharpoonup ', \, \sigma ', \, \cdot ' \bigl)$
be two  algebraic extending structures of $A$ through $V$ and $A{}_{\sigma}\#_{}V$, respectively $A{}_{\sigma'}\#_{}  V$ the  associated unified
products. Then there exists a bijection between the set of all
homomorphisms of algebras $\psi: A{}_{\sigma}\#_{}V\to A{}_{\sigma'}\#_{}  V$which
stabilize $A$ and the set of pairs $(r, \, v)$, where $r: V \to
A$, $s: V \to V$ are linear maps satisfying the following
compatibility conditions for any $x \in A$, $u$, $v \in V$:
\begin{enumerate}
\item[(M1)] $r(x \cdot y) = r(x)r(y) + \sigma ' (s(x), s(y)) - \sigma(x, y) + r(x) \ppl' s(y) + s(x) \ppr' r(y)$,
\item[(M2)] $s(x \cdot y) = r(x) \trr ' s(y) + s(x)\trl ' r(y) + s(x) \cdot ' s(y)$,
 \item[(M3)] $r(x\trl  {a}) = r(x) {a} - x \ppr {a} + s(x) \ppr' {a}$,
  \item[(M5)] $r({a} \trr x) = {a}r(x) - {a}\ppl x + {a} \ppl' s(x)$,
 \item[(M4)] $s(x\trl {a}) = s(x)\trl' {a}$,
 \item[(M6)] $s({a}\trr x) = {a} \trr' s(x)$.
\end{enumerate}
Under the above bijection the homomorphism of algebras $\varphi =\varphi _{(r, s)}: A_{\sigma}\# {}_{}H \to A_{\sigma'}\# {}_{}H$ corresponding to
$(r, s)$ is given for any $a\in A$ and $x \in V$ by:
$$\varphi(a, x) = (a + r(x), s(x))$$
Moreover, $\varphi  = \varphi _{(r, s)}$ is an isomorphism if and only if $s: V \to V$ is an isomorphism linear map.
\end{lemma}

Let ${A}$ be a BiHom-algebra and $V$ a BiHom-vector space. Two algebraic extending systems $\Omega^{(i)}({A}, V)$ and ${\Omega'^{(i)}}({A}, V)$  are called equivalent if $\varphi_{r,s}$ is an isomorphism.  We denote it by $\Omega^{(i)}({A}, V)\equiv{\Omega'^{(i)}}({A}, V)$.
From the above Lemmas \ref{lem:33-1}, \ref{lem:33-100}, \ref{lem:33-2} and \ref{lem:33-200}, we obtain the following Theorem \ref{thm3-1}.

\begin{theorem}\label{thm3-1}
Let ${A}$ be a BiHom-algebra, $E$ a BiHom-vector space containing ${A}$ as a subspace and
$V$ be a complement of ${A}$ in $E$.
Denote $\mathcal{HA}(V,{A}):=\mathcal{A}^{(1)}({A},V)\sqcup \mathcal{A}^{(2)}({A},V) /\equiv$. Then the map
\begin{eqnarray}
&&\Psi: \mathcal{HA}(V,{A})\rightarrow Extd(E,{A}),\\
&&\overline{\Omega^{(1)}({A},V)}\mapsto A_{}\#_{\theta} V,\quad \overline{\Omega^{(2)}({A},V)}\mapsto A_{\sigma}\# {}_{} V
\end{eqnarray}
is bijective, where $\overline{\Omega^{(i)}({A}, V)}$ is the equivalence class of $\Omega^{(i)}({A}, V)$ under $\equiv$.
\end{theorem}

Next we consider the BiHom-coalgebra structures on $E=A^{P}\# {}^{Q}V$.
There are two cases for $(A,\Delta_A)$ to be a  coalgebra. The first case is  when $Q=0$,  then we obtain the following type (c1) unified product for  coalgebras.
\begin{lemma}\label{cor01}
Let $({A},\Delta_A)$ be a BiHom- coalgebra and $V$ a BiHom-vector space.
An  extending datum  of ${A}$ by $V$ of  type (c1) is  $\Omega^c({A},V)=(\phi, {\psi},\rho,\gamma, P, \Delta_V)$ with  linear maps
\begin{eqnarray*}
&&\phi: A \to V \otimes A,\quad  \psi: A \to A\otimes V,\\
&&\rho: V  \to A\otimes V,\quad  \gamma: V \to V \otimes A,\\
&& {P}: A\rightarrow {V}\otimes {V},\quad\Delta_V: V\rightarrow V\otimes V.
\end{eqnarray*}
 Denote by $A^{P}\# {}^{} V$ the vector space $E={A}\oplus V$ with the linear map
$\Delta_E: E\rightarrow E\otimes E$ given by
$$\Delta_{E}(a,x)=(\Delta_{A}+\phi+\psi+P)(a)+(\Delta_{V}+\rho+\gamma)(x), $$
that is
$$\Delta_{E}(a)= a\li \ot a\lii+ a\moi \ot a\mo+a\mo\ot a\mi+a\ppi\ot a\pii,$$
$$\Delta_{E}(x)= x\li \ot x\lii+ x\boi \ot x\boo+x\boo \ot x\bi.$$
Then $A^{P}\# {}^{} V$  is a  BiHom-coalgebra with the comultiplication given above if and only if the following compatibility conditions hold:
\begin{enumerate}
\item[(C1)] $\alvv(a_{(-1)}) \otimes \Delta_{A}\left(a_{(0)}\right)=\phi\left(a_{1}\right) \otimes \beaa(a_{2})+\gamma\left(a_{(-1)}\right) \otimes \beaa(a_{(0)})$,
\item[(C2)] $\Delta_{A}\left(a_{(0)}\right) \otimes \bevv(a_{(1)})=\alaa(a_{1}) \otimes \psi\left(a_{2}\right)+\alaa(a_{(0)} )\otimes \rho\left(a_{(1)}\right)$,
\item[(C3)] $\alaa(x_{[-1]} )\otimes \Delta_{V}\left(x_{[0]}\right)=\rho\left(x_{1}\right) \otimes \bevv(x_{2})+\psi\left(x_{[-1]}\right) \otimes \bevv(x_{[0]})$,
\item[(C4)] $\Delta_{V}(x\boo)\ot \beaa(x\bi)=\alvv(x\boo)\ot \phi(x\bi)+\alvv(x_1)\ot \gamma(x_2)$,
\item[(C5)] $\alaa(a_{1}) \otimes \phi\left(a_{2}\right)+\alaa(a_{(0)}) \otimes \gamma\left(a_{(1)}\right)=\psi\left(a_{1}\right) \otimes \beaa(a_{2})+\rho\left(a_{(-1)}\right) \otimes \beaa(a_{(0)})$,
\item[(C6)] $\alvv(x_{1}) \otimes \rho\left(x_{2}\right)+\alvv(x_{[0]} )\otimes \psi\left(x_{[1]}\right)=\phi\left(x_{[-1]}\right) \otimes \bevv(x_{[0]})+\gamma\left(x_{1}\right) \otimes \bevv( x_{2})$,
\item[(C7)] $\alvv(a_{(-1)}) \otimes \phi\left(a_{(0)}\right)+\alvv(a\ppi)\otimes \gamma\left(a\pii\right)=\Delta_{V}\left(a_{(-1)}\right) \otimes \beaa(a_{(0)})+P\left(a_{1}\right) \otimes \beaa(a_{2})$,
\item[(C8)] $\alaa(a_{(0)})\ot \Delta_{V}\left(a_{(1)}\right)+\alaa(a\ppi)\otimes P\left(a\pii\right)\\
=\psi\left(a_{(0)}\right) \otimes \bevv(a_{(1)})+\rho\left(a\ppi\right) \otimes \bevv(a\pii)$,
\item[(C9)] $\alvv(a_{(-1)} )\otimes \psi\left(a_{(0)}\right)+\alvv(a\ppi )\otimes \rho\left(a\pii\right)=\phi\left(a_{(0)}\right) \otimes \bevv(a_{(1)})+\gamma\left(a\ppi\right) \otimes \bevv(a\pii)$,
\item[(C10)] $\alaa(x_{[-1]}) \otimes \rho\left(x_{[0]}\right)=\Delta_{A}\left(x_{[-1]}\right) \otimes \bevv(x_{[0]})$,
\item[(C11)] $\alvv(x\boo)\ot\Delta_A(x\bi)=\gamma(x\boo)\ot \beaa(x\bi)$,
\item[(C12)] $\alaa(x\boi)\ot\gamma(x\boo)=\rho(x\boo)\ot \beaa(x\bi)$.
\end{enumerate}
\end{lemma}
Denote the set of all  coalgebraic extending datum of ${A}$ by $V$ of type (c1) by $\mathcal{C}^{(1)}({A},V)$.

\begin{lemma}\label{lem:33-3}
Let $({A},\Delta_A)$ be a  BiHom-coalgebra and $E$ a BiHom-vector space containing ${A}$ as a subspace. Suppose that there is a  BiHom-coalgebra structure $(E,\Delta_E)$ on $E$ such that  $p: E\to {A}$ is a  BiHom-coalgebra homomorphism. Then there exists a  BiHom-coalgebraic extending system $\Omega^c({A}, V)$ of $({A},\Delta_A)$ by $V$ such that $(E,\Delta_E)\cong A^{P}\# {}^{} V$.
\end{lemma}

\begin{proof}
Let $p: E\to {A}$ and $\pi: E\to V$ be the projection map and $V=\ker({p})$.
Then the extending datum of $({A},\Delta_A)$ by $V$ is defined as follows:
\begin{eqnarray*}
&&{\phi}: A\rightarrow V\ot {A},~~~~{\phi}(a)=(\pi\otimes {p})\Delta_E(a),\\
&&{\psi}: A\rightarrow A\ot V,~~~~{\psi}(a)=({p}\otimes \pi)\Delta_E(a),\\
&&{\rho}: V\rightarrow A\ot V,~~~~{\rho}(x)=({p}\otimes \pi)\Delta_E(x),\\
&&{\gamma}: V\rightarrow V\ot {A},~~~~{\gamma}(x)=(\pi\otimes {p})\Delta_E(x),\\
&&\Delta_V: V\rightarrow V\otimes V,~~~~\Delta_V(x)=(\pi\otimes \pi)\Delta_E(x),\\
&&Q: V\rightarrow {A}\otimes {A},~~~~Q(x)=({p}\otimes {p})\Delta_E(x)\\
&&P: A\rightarrow {V}\otimes {V},~~~~P(a)=({\pi}\otimes {\pi})\Delta_E(a).
\end{eqnarray*}
One check that  $\varphi: A^{P}\# {}^{} V\to E$ given by $\varphi(a,x)=a+x$ for all $a\in A, x\in V$ is a  coalgebra isomorphism.
\end{proof}

\begin{lemma}\label{lem-c1}
Let $\Omega^{(1)}({A}, V)=(\phi, {\psi},\rho,\gamma, P, \Delta_V)$ and ${\Omega'^{(1)}}({A}, V)=(\phi', {\psi'},\rho',\gamma',  P', \Delta'_V)$ be two  coalgebraic extending datums of $({A},\Delta_A)$ by $V$. Then there exists a bijection between the set of   coalgebra homomorphisms $\varphi: A^{P}\# {}^{} V\rightarrow A^{P'}\# {}^{} V$ whose restriction on ${A}$ is the identity map and the set of pairs $(r,s)$, where $r:V\rightarrow {A}$ and $s:V\rightarrow V$ are two linear maps satisfying
\begin{eqnarray}
\label{comorph11}&&P'(a)=s(a\ppi)\ot s(a\pii),\\
\label{comorph121}&&\phi'(a)={s}(a\lmoi)\ot a\lmo+s(a\ppi)\ot r(a\pii),\\
\label{comorph122}&&\psi'(a)=a\lmo\ot {s}(a\mi) +r(a\ppi)\ot s(a\pii),\\
\label{comorph13}&&\Delta'_A(a)=\Delta_A(a)+{r}(a\lmoi)\ot a\lmo+a\lmo\ot {r}(a\mi)+r(a\ppi)\ot r(a\pii)\\
\label{comorph21}&&\Delta_V'({s}(x))=({s}\otimes {s})\Delta_V(x),\\
\label{comorph221}&&{\rho}'({s}(x))=s(x\li)\ot r(x\lii)+x\boi\ot s(x\boo),\\
\label{comorph222}&&{\gamma}'({s}(x))=s(x\li)\ot r(x\lii)+s(x\boo)\ot x\bi,\\
\label{comorph23}&&\Delta'_A({r}(x))=r(x\li)\ot r(x\lii)+x\boi\ot r(x\boo)+r(x\boo)\ot x\bi.
\end{eqnarray}
Under the above bijection the  coalgebra homomorphism $\varphi=\varphi_{r,s}: A^{P}\# {}^{} V\rightarrow A^{P'}\# {}^{} V$ to $(r,s)$ is given by $\varphi(a,x)=(a+r(x),s(x))$ for all $a\in {A}$ and $x\in V$. Moreover, $\varphi=\varphi_{r,s}$ is an isomorphism if and only if $s: V\rightarrow V$ is a linear isomorphism.
\end{lemma}

\begin{proof}
Let $\varphi: A^{P}\# {}^{} V\rightarrow A^{P'}\# {}^{} V$  be a  BiHom-coalgebra homomorphism  whose restriction on ${A}$ is the identity map. Then $\varphi$ is determined by two linear maps $r: V\rightarrow {A}$ and $s: V\rightarrow V$ such that
$\varphi(a,x)=(a+r(x),s(x))$ for all $a\in {A}$ and $x\in V$. We will prove that
$\varphi$ is a homomorphism of  coalgebras if and only if the above condtions hold.
First we it easy to see that  $\Delta'_E\varphi(a)=(\varphi\otimes \varphi)\Delta_E(a)$ for all $a\in {A}$.
\begin{eqnarray*}
\Delta'_E\varphi(a)&=&\Delta'_E(a)=\Delta'_A(a)+\phi'(a)+\psi'(a)+P'(a),
\end{eqnarray*}
and
\begin{eqnarray*}
&&(\varphi\otimes \varphi)\Delta_E(a)\\
&=&(\varphi\otimes \varphi)\left(\Delta_A(a)+\phi(a)+\psi(a)+P(a)\right)\\
&=&\Delta_A(a)+{r}(a\lmoi)\ot a\lmo+{s}(a\lmoi)\ot a\lmo+a\lmo\ot {r}(a\mi) +a\lmo\ot {s}(a\mi)\\
&&+r(a\ppi)\ot r(a\pii)+r(a\ppi)\ot s(a\pii)+s(a\ppi)\ot r(a\pii)+s(a\ppi)\ot s(a\pii).
\end{eqnarray*}
Thus we obtain that $\Delta'_E\varphi(a)=(\varphi\otimes \varphi)\Delta_E(a)$  if and only if the conditions \eqref{comorph11}, \eqref{comorph121}, \eqref{comorph122} and \eqref{comorph13} hold.
Then we consider that $\Delta'_E\varphi(x)=(\varphi\otimes \varphi)\Delta_E(x)$ for all $x\in V$.
\begin{eqnarray*}
\Delta'_E\varphi(x)&=&\Delta'_E({r}(x),{s}(x))=\Delta'_E({r}(x))+\Delta'_E({s}(x))\\
&=&\Delta'_A({r}(x))+\Delta'_V({s}(x))+{\rho}'({s}(x))+{\gamma}'({s}(x))),
\end{eqnarray*}
and
\begin{eqnarray*}
&&(\varphi\otimes \varphi)\Delta_E(x)\\
&=&(\varphi\otimes \varphi)(\Delta_V(x)+{\rho}(x)+{\gamma}(x))\\
&=&(\varphi\otimes \varphi)(x\li\ot x\lii+x\boi\ot x\boo+x\boo\ot x\bi)\\
&=&r(x\li)\ot r(x\lii)+r(x\li)\ot s(x\lii)+s(x\li)\ot r(x\lii)+s(x\li)\ot s(x\lii)\\
&&+x\boi\ot r(x\boo)+x\boi\ot s(x\boo)+r(x\boo)\ot x\bi+s(x\boo)\ot x\bi.
\end{eqnarray*}
Thus we obtain that $\Delta'_E\varphi(x)=(\varphi\otimes \varphi)\Delta_E(x)$ if and only if the conditions  \eqref{comorph21},  \eqref{comorph221},  \eqref{comorph222} and \eqref{comorph23}  hold. By definition, we obtain that $\varphi=\varphi_{r,s}$ is an isomorphism if and only if $s: V\rightarrow V$ is a linear isomorphism.
\end{proof}

The second case is $\phi=0$ and  $\psi=0$, we obtain  the following type (c2) unified product for  coalgebras.
\begin{lemma}\label{cor02}
Let $({A},\Delta_A)$ be a  BiHom-coalgebra and $V$ a BiHom-vector space.
An  extending datum  of $({A},\Delta_A)$ by $V$ of type (c2)  is  $\Omega^{(2)}({A},V)=(\rho, \gamma, {Q}, \Delta_V)$ with  linear maps
\begin{eqnarray*}
&&\rho: V  \to A\otimes V,\quad  \gamma: V \to V \otimes A,\quad \Delta_{V}: V \to V\otimes V,\quad Q: V \to A\otimes A.
\end{eqnarray*}
 Denote by $A^{}\# {}^{Q} V$ the BiHom-vector space $E={A}\oplus V$ with the comultiplication
$\Delta_E: E\rightarrow E\otimes E$ given by
\begin{eqnarray*}
\Delta_{E}(a,x)&=&\Delta_{A}(a)+(\Delta_{V}+\rho+\gamma+Q)(x),
\end{eqnarray*}
that is
$$\Delta_{E}(a)= a\li \ot a\lii,\quad \Delta_{E}(x)= x\li \ot x\lii+ x\boi \ot x\boo+x\boo \ot x\bi+x\qi\ot x\qii.$$
Then $A^{}\# {}^{Q} V$  is a  BiHom-coalgebra with the comultiplication given above if and only if the following compatibility conditions hold:
\begin{enumerate}
\item[(D1)] $\alaa(x_{[-1]}) \otimes \Delta_{V}\left(x_{[0]}\right)=\rho\left(x_{1}\right) \otimes \bevv(x_{2})$,
\item[(D2)] $\Delta_{V}(x\boo)\ot \beaa(x\bi)=\alvv(x_1)\ot \gamma(x_2)$,
\item[(D3)] $\alvv(x_{1}) \otimes \rho\left(x_{2}\right)=\gamma\left(x_{1}\right) \otimes \bevv(x_{2})$,
\item[(D4)] $\alaa(x_{[-1]}) \otimes \rho\left(x_{[0]}\right)=\Delta_{A}\left(x_{[-1]}\right) \otimes \bevv(x_{[0]})+Q\left(x_{1}\right) \otimes \bevv(x_{2})$,
\item[(D5)] $\gamma(x\boo)\ot \beaa(x\bi)=\alvv(x\boo)\ot\Delta_A(x\bi)+\alvv(x_1)\ot Q(x_2)$,
\item[(D6)] $\alaa(x\boi)\ot\gamma(x\boo)=\rho(x\boo)\ot \beaa(x\bi)$,
\item[(D7)]  $ \alaa(x\qi)\ot \Delta_A(x\qii)+\alaa(x\boi)\ot Q(x\boo)=\Delta_A(x\qi)\ot \beaa(x\qii)+Q(x\boo)\ot \beaa(x\bi)$,
\item[(D8)] $\alvv(x_1)\ot \Delta_V(x_2)=\Delta_V(x\li)\ot \bevv( x\lii)$.
\end{enumerate}
In this case $(V,\Delta_V)$ is a BiHom-coalgebra.
\end{lemma}

Denote the set of all  coalgebraic extending datum of ${A}$ by $V$ of type (c2) by $\mathcal{C}^{(2)}({A},V)$.

Similar to the BiHom-algebra case,  one  show that any BiHom-coalgebra structure on $E$ containing ${A}$ as a  subcoalgebra is isomorphic to such a unified coproduct.
\begin{lemma}\label{lem:33-4}
Let $({A},\Delta_A)$ be a  BiHom-coalgebra and $E$ a BiHom-vector space containing ${A}$ as a subspace. Suppose that there is a  BiHom-coalgebra structure $(E,\Delta_E)$ on $E$ such that  $({A},\Delta_A)$ is a subcoalgebra of $E$. Then there exists a  BiHom-coalgebraic extending system $\Omega^{(2)}({A}, V)$ of $({A},\Delta_A)$ by $V$ such that $(E,\Delta_E)\cong A^{}\# {}^{Q} V$.
\end{lemma}

\begin{proof}
Let $p: E\to {A}$ and $\pi: E\to V$ be the projection map and $V=ker({p})$.
Then the extending datum of $({A},\Delta_A)$ by $V$ is defined as follows:
\begin{eqnarray*}
&&{\rho}: V\rightarrow A\ot V,~~~~{\phi}(x)=(p\otimes {\pi})\Delta_E(x),\\
&&{\gamma}: V\rightarrow V\ot {A},~~~~{\phi}(x)=(\pi\otimes {p})\Delta_E(x),\\
&&\Delta_V: V\rightarrow V\otimes V,~~~~\Delta_V(x)=(\pi\otimes \pi)\Delta_E(x),\\
&&Q: V\rightarrow {A}\otimes {A},~~~~Q(x)=({p}\otimes {p})\Delta_E(x).
\end{eqnarray*}
One check that  $\varphi: A^{}\# {}^{Q} V\to E$ given by $\varphi(a,x)=a+x$ for all $a\in A, x\in V$ is a  coalgebra isomorphism.
\end{proof}

\begin{lemma}\label{lem-c2}
Let $\Omega^{(2)}({A}, V)=(\rho, \gamma, {Q}, \Delta_V)$ and ${\Omega'^{(2)}}({A}, V)=(\rho', \gamma', {Q'}, \Delta'_V)$ be two  coalgebraic extending datums of $({A},\Delta_A)$ by $V$. Then there exists a bijection between the set of   coalgebra homomorphisms $\varphi: A \# {}^{Q} V\rightarrow A \# {}^{Q'} V$ whose restriction on ${A}$ is the identity map and the set of pairs $(r,s)$, where $r:V\rightarrow {A}$ and $s:V\rightarrow V$ are two linear maps satisfying
\begin{eqnarray}
\label{comorph1}&&{\rho}'({s}(x))=r(x\li)\ot s(x\lii)+x\boi\ot s(x\boo),\\
\label{comorph2}&&{\gamma}'({s}(x))=s(x\li)\ot r(x\lii)+s(x\boo)\ot x\bi,\\
\label{comorph3}&&\Delta_V'({s}(x))=({s}\otimes {s})\Delta_V(x)\\
\label{comorph4}&&\Delta'_A({r}(x))+{Q'}({s}(x))=r(x\li)\ot r(x\lii)+x\boi\ot r(x\boo)+r(x\boo)\ot x\bi+{Q}(x).
\end{eqnarray}
Under the above bijection the  coalgebra homomorphism $\varphi=\varphi_{r,s}: A^{ }\# {}^{Q} V\rightarrow A^{ }\# {}^{Q'} V$ to $(r,s)$ is given by $\varphi(a,x)=(a+r(x),s(x))$ for all $a\in {A}$ and $x\in V$. Moreover, $\varphi=\varphi_{r,s}$ is an isomorphism if and only if $s: V\rightarrow V$ is a linear isomorphism.
\end{lemma}

\begin{proof} The proof is similar as the proof of Lemma \ref{lem-c1}.
Let $\varphi: A^{ }\# {}^{Q} V\rightarrow A^{}\# {}^{Q'} V$  be a  coalgebra homomorphism  whose restriction on ${A}$ is the identity map.
First we it easy to see that  $\Delta'_E\varphi(a)=(\varphi\otimes \varphi)\Delta_E(a)$ for all $a\in {A}$.
Then we consider that $\Delta'_E\varphi(x)=(\varphi\otimes \varphi)\Delta_E(x)$ for all $x\in V$.
\begin{eqnarray*}
\Delta'_E\varphi(x)&=&\Delta'_E({r}(x),{s}(x))=\Delta'_E({r}(x))+\Delta'_E({s}(x))\\
&=&\Delta'_A({r}(x))+\Delta'_V({s}(x))+{\rho}'({s}(x))+{\gamma}'({s}(x))+{Q}'({s}(x)),
\end{eqnarray*}
and
\begin{eqnarray*}
&&(\varphi\otimes \varphi)\Delta_E(x)\\
&=&(\varphi\otimes \varphi)(\Delta_V(x)+{\rho}(x)+{\gamma}(x)+{Q}(x))\\
&=&(\varphi\otimes \varphi)(x\li\ot x\lii+x\boi\ot x\boo+x\boo\ot x\bi+{Q}(x))\\
&=&r(x\li)\ot r(x\lii)+r(x\li)\ot s(x\lii)+s(x\li)\ot r(x\lii)+s(x\li)\ot s(x\lii)\\
&&+x\boi\ot r(x\boo)+x\boi\ot s(x\boo)+r(x\boo)\ot x\bi+s(x\boo)\ot x\bi+{Q}(x).
\end{eqnarray*}
Thus we obtain that $\Delta'_E\varphi(x)=(\varphi\otimes \varphi)\Delta_E(x)$ if and only if the conditions \eqref{comorph1}, \eqref{comorph2},  \eqref{comorph3} and \eqref{comorph4} hold. By definition, we obtain that $\varphi=\varphi_{r,s}$ is an isomorphism if and only if $s: V\rightarrow V$ is a linear isomorphism.
\end{proof}

Let $({A},\Delta_A)$ be a  BiHom-coalgebra and $V$ a BiHom-vector space. Two  coalgebraic extending systems $\Omega^{(i)}({A}, V)$ and ${\Omega'^{(i)}}({A}, V)$  are called equivalent if $\varphi_{r,s}$ is an isomorphism.  We denote it by $\Omega^{(i)}({A}, V)\equiv{\Omega'^{(i)}}({A}, V)$.
From the above Lemmas \ref{lem:33-3}, \ref{lem-c1}, \ref{lem:33-4} and \ref{lem-c2}, we obtain the following Theorem \ref{thm3-2}.
\begin{theorem}\label{thm3-2}
Let $({A},\Delta_A)$ be a  BiHom-coalgebra, $E$ a BiHom-vector space containing ${A}$ as a subspace and
$V$ be a ${A}$-complement in $E$. Denote $\mathcal{HC}(V,{A}):=\mathcal{C}^{(1)}({A},V)\sqcup\mathcal{C}^{(2)}({A},V) /\equiv$. Then the map
\begin{eqnarray*}
&&\Psi: \mathcal{HC}_{{A}}^2(V,{A})\rightarrow CExtd(E,{A}),\\
&&\overline{\Omega^{(1)}({A},V)}\mapsto A^{P}\# {}^{} V,
 \quad \overline{\Omega^{(2)}({A},V)}\mapsto A^{}\# {}^{Q} V
\end{eqnarray*}
is bijective, where $\overline{\Omega^{(i)}({A},V)}$ is the equivalence class of $\Omega^{(i)}({A}, V)$ under $\equiv$.
\end{theorem}

\subsection{Extending structures for BiHom-Frobenius algebras}
Let $(A,\cdot,\Delta_A)$ be a BiHom-Frobenius algebra. From Definition \ref{bi-cycle1} we have the following two cases.

The first case is that we assume $Q=0$ and $\ppr, \ppl$ to be trivial. Then by the above Theorem \ref{main2}, we obtain the following result.

\begin{theorem}\label{thm-41}
Let $(A,\cdot,\Delta_A)$ be a BiHom-Frobenius algebra and $V$ a BiHom-vector space.
An extending datum of ${A}$ by $V$ of type (I) is  $\Omega^{(1)}({A},V)=(\trr, \trl, \phi, \psi, P, \cdot_V, \Delta_V)$ consisting of  linear maps
\begin{eqnarray*}
\trr: V\otimes A\rightarrow V, \quad\trl: A\otimes V\rightarrow V,~~~~\theta:  A\otimes A \rightarrow {V},~~~\cdot_V:V\otimes V \rightarrow V,\\
 \phi :A \to V\otimes A, \quad{\psi}: V\to  V\otimes A,~~~~{P}: A\rightarrow {V}\otimes {V},~~~~\Delta_V: V\rightarrow V\otimes V.
\end{eqnarray*}
Then the unified product $A^{P}_{}\# {}^{}_{\theta}\, V$ with multiplication
\begin{eqnarray}
(a, x) (b, y)=(ab, \, xy+ a\trr y+x\trl b+\theta(a, b))
\end{eqnarray}
and comultiplication
\begin{eqnarray}
\Delta_E(a)=\Delta_A(a)+{\phi}(a)+{\psi}(a)+P(a),\quad \Delta_E(x)=\Delta_V(x)+{\rho}(x)+{\gamma}(x)
\end{eqnarray}
form an infinitesimal bialgebra if and only if $A_{}\# {}_{\theta} V$ form a BiHom-algebra, $A^{P}\# {}^{} \, V$ form a   coalgebra and the following conditions are satisfied:
\begin{enumerate}
\item[(E1)]  $\phi(a b)+\gamma(\theta(a, b)) =\alvv(a_{(-1)}) \otimes\left(a_{(0)} \beaa(b)\right)$\\
$=(\alaa(a) \trr b_{(-1)}) \otimes \beaa(b_{(0)})+\theta\left(\alaa(a), b_{1}\right) \otimes \beaa(b_{2})$,
\item[(E2)] $\psi(a b)+\rho(\theta(a, b))=(\alaa(a) b_{(0)}) \otimes \bevv(b_{(1)})$\\
$=\alaa(a_{(0)}) \otimes\left(a_{(1)} \trl \beaa(b)\right)+\alaa(a_{1}) \otimes \theta\left(a_{2}, \beaa(b)\right)$,
\item[(E3)] $\rho(x y)+\psi(\sigma(x, y))=\alaa(x_{[-1]}) \otimes x_{[0]} \bevv(y) +\alaa(x\qi)\otimes x\qii\trr \bevv(y)$,
\item[(E4)] $\gamma(x y)=\alvv(x)y_{[0]}\otimes \beaa(y_{[1]})$,
\item[(E5)] $\Delta_{V}(a \trr y)=\alvv (a_{(-1)} )\otimes\left(a_{(0)}\trr\bevv( y)\right)+\alvv(a\ppi) \otimes a\pii \bevv(y)\\
=\left(\alaa(a) \trr y_{1}\right) \otimes \bevv(y_{2})+\theta\left(\alaa(a), y_{[-1]}\right) \otimes \bevv(y_{[0]})$,
\item[(E6)] $\Delta_{V}(x \trl b)=\alvv(x_{1}) \otimes\left(x_{2}  \trl \beaa( b)\right)+\alvv(x_{[0]}) \otimes \theta\left(x_{[1]}, \beaa(b)\right)\\
=\left(\alvv(x)\trl b_{(0)}\right) \otimes \bevv( b_{(1)})+\alvv(x) b\ppi\otimes \bevv(b\pii)$,
\item[(E7)]$\Delta_{H}(\theta(a,b))+P(a, b)$\\
$=\alhh(a_{(-1)}) \otimes\theta(a_{(0)},\beaa(b))+a\ppi \otimes \alhh(a\pii)\trl \beaa(b)\\
=\theta(\alaa(a),b_{(0)})\otimes \behh(b_{(1)})+\alaa(a)\trr b\ppi\otimes \behh(b\pii)$,
\item[(E8)]
 $\gamma(x\trl b)=\alhh(x_{[0]}) \otimes x_{[1]} \beaa(b)\\
  =\alhh(x) b_{(-1)} \otimes \beaa(b_{(0)})+\left(\alhh(x)\trl b_{1}\right) \otimes \beaa(b_{2})$,
\item[(E9)]
$\rho(a \trr y)=\alaa( a)y_{[-1]} \otimes \behh(y_{[0]})\\
=\alaa(a_{(0)}) \otimes\left(a_{(1)}  \behh(y)\right)+\alaa(a_{1}) \otimes\left(a_{2} \trr \behh(y)\right)$,
\item[(E10)]
$\rho(x\trl b)=\alaa(x_{[-1]}) \otimes\left(x_{[0]} \trl \beaa(b)\right)$,
\item[(E11)]
$\gamma(a \trr y)=\left(\alaa(a)\trr y_{[0]}\right)\otimes \beaa(y_{[1]})$,
\item[(E12)] $\Delta_{V}(x y)=\alvv(x_{1}) \otimes x_{2} \bevv(y)+\alvv(x_{[0]}) \otimes\left(x_{[1]} \trr \bevv( y)\right)\\
=\alvv(x) y_{1} \otimes \bevv(y_{2})+\left(\alvv(x) \trl y_{[-1]}\right) \otimes \bevv(y_{[0]}).$
\end{enumerate}
Conversely, any BiHom-Frobenius algebra structure on $E$ with the canonical projection map $p: E\to A$ both a BiHom-algebra homomorphism and a BiHom-coalgebra homomorphism is of this form.
\end{theorem}
Note that in this case, $(V,\cdot,\Delta_V)$ is a  braided BiHom-Frobenius algebra. Although $(A,\cdot,\Delta_A)$ is not a  sub-bialgebra of $E=A^{P}_{}\# {}^{}_{\theta}\, V$, it is indeed a BiHom-Frobenius algebra and a subspace $E$.
Denote the set of all BiHom-Frobenius algebraic extending datum of type (I) by $\mathcal{IB}^{(1)}({A},V)$.

The second case is that we assume $P=0, \theta=0$ and $\phi, \psi$ to be trivial. Then by the above Theorem \ref{main2}, we obtain the following result.

\begin{theorem}\label{thm-42}
Let $A$ be an infinitesimal bialgebra and $V$ a BiHom-vector space.
An extending datum of ${A}$ by $V$ of type (II) is  $\Omega^{(2)}({A},V)=(\ppr, \ppl, \trr, \trl, \sigma, \rho, \gamma, Q,  \cdot_V, \Delta_V)$ consisting of  linear maps
\begin{eqnarray*}
\trl: V\otimes {A}\rightarrow {V},~~~~\trr: A\otimes {V}\rightarrow V,~~~\ppr: V\otimes A\to A,~~~~\ppl: A\otimes V\to A,\\
{\rho}: V\to  A\otimes V,~~~~{\gamma}: V\to  V\otimes A,~~~~\sigma:  V\otimes V \rightarrow {A},~~~~{Q}: V\rightarrow {A}\otimes {A},\\
\end{eqnarray*}
Then the unified product $A^{}_{\sigma}\# {}^{Q}_{}\, V$ with multiplication
\begin{eqnarray}
(a, x)(b, y)=\big(ab+x\ppr b+a\ppl y+\sigma(x, y), \, xy+x\trl b+a\trr y\big)
\end{eqnarray}
and comultiplication
\begin{eqnarray}
\Delta_E(a)=\Delta_A(a),\quad \Delta_E(x)=\Delta_V(x)+{\rho}(x)+{\gamma}(x)+Q(x)
\end{eqnarray}
form an infinitesimal bialgebra if and only if $A_{\sigma}\# {}_{} V$ forms a BiHom-algebra, $A^{}\# {}^{Q}V$ forms a   coalgebra and the following conditions are satisfied:
\begin{enumerate}
\item[(F1)] $\rho(x y)=\alaa(x_{[-1]}) \otimes x_{[0]}\bevv( y)+\alaa(x\qi)\otimes x\qii\trr \bevv(y)\\
  =\left(\alvv(x) \ppr y_{[-1]}\right) \otimes \bevv(y_{[0]})+\sigma\left(\alvv(x), y_{1}\right) \otimes \bevv(y_{2})$,
\item[(F2)] $\gamma(x y)=\alvv(x_{[0]})\otimes (x_{[1]}\ppl \bevv(y))+\alvv(x_{1}) \otimes \sigma\left(x_{2}, \bevv(y)\right)\\
   =\alvv(x)y_{[0]}\otimes \beaa(y_{[1]})+\alvv(x)\trl y\qi\otimes \beaa(y\qii)$,
\item[(F3)] $\Delta_{A}(x \ppr b)+Q(x\trl  b)=\left(\alvv(x) \ppr b_{1}\right) \otimes \beaa(b_{2})$\\
$=\alaa(x_{[-1]}) \otimes\left(x_{[0]} \ppr \beaa(b)\right)+\alaa(x\qi) \otimes x\qii \beaa(b)$,

\item[(F4)] $\Delta_{A}(a\ppl y)+Q(a\trr y)$\\
$=\alaa(a_{1}) \otimes\left(a_{2} \ppl \bevv(y)\right)+\alaa(a_{(0)}) \otimes \sigma\left(a_{(1)}, \bevv(y)\right)\\
  =\left(\alaa(a\ppl) y_{[0]}\right) \otimes \beaa(y_{[1]})+\alaa(a) y\qi \otimes \beaa(y\qii)$,

\item[(F5)] $\Delta_{V}(a \trr y)=\left(\alaa(a) \trr y_{1}\right) \otimes \bevv(y_{2})$,

\item[(F6)] $\Delta_{V}(x \trl b)=\alvv(x_{1}) \otimes\left(x_{2}  \trl \beaa(b)\right)$,

\item[(F7)]$\Delta_{A}(\sigma(x,y))+Q(x, y)$\\
$=\alhh(x_{[-1]})\otimes \sigma(x_{[0]},\behh(y))+\alaa(x\qi)\otimes x\qii\ppl \behh(y)$\\
$=\sigma(\alhh(x),y_{[0]})\otimes \beaa(y_{[-1]})+\alhh(x)\ppr y\qi\otimes \beaa(y\qii)$,

\item[(F8)]
 $\gamma(x\trl b)=\left(\alhh(x)\trl b_{1}\right) \otimes \beaa(b_{2})$\\
 $=\alhh(x_{1}) \otimes\left(x_{2} \ppr\beaa( b)\right)+\alhh(x_{[0]}) \otimes x_{[1]} \beaa(b)$,

\item[(F9)]
$\rho(a \trr y)=\alaa(a_{1}) \otimes\left(a_{2} \trr \behh(y)\right)$\\
$=\left(\alaa(a)\ppl y_{1}\right) \otimes \behh(y_{2})+\alaa(a) y_{[-1]} \otimes \behh(y_{[0]})$,

\item[(F10)]
$\rho(x\trl b)=\alaa(x_{[-1]}) \otimes\left(x_{[0]} \trl \beaa(b)\right)$,

\item[(F11)]
$\gamma(a \trr y)=\left(\alaa(a)\trr y_{[0]}\right)\otimes \beaa(y_{[1]})$,

\item[(F12)] $\Delta_{V}(x y)=\alvv(x_{1}) \otimes x_{2} \bevv(y)+\alvv(x_{[0]}) \otimes\left(x_{[1]} \trr \bevv( y)\right)\\
=\alvv(x) y_{1} \otimes \bevv(y_{2})+\left(\alvv(x) \trl y_{[-1]}\right) \otimes \bevv(y_{[0]}).$
\end{enumerate}
Conversely, any BiHom-Frobenius algebra structure on $E$ with the canonical injection map $i: A\to E$ both a BiHom-algebra homomorphism and a   coalgebra homomorphism is of this form.
\end{theorem}
Note that in this case, $(A,\cdot,\Delta_A)$ is a  sub-algebra and sub-coalgebra of $E=A^{}_{\sigma}\# {}^{Q}_{}\, V$ and $(V,\cdot,\Delta_V)$ is a  braided BiHom-Frobenius algebra.
Denote the set of all  BiHom-Frobenius algebraic extending datum of type (II) by $\mathcal{IB}^{(2)}({A},V)$.

In the above two cases, we find that  the braided BiHom-Frobenius algebra $V$ plays an important role in the extending problem of BiHom-Frobenius algebra $A$.
Note that $A^{P}_{}\# {}^{}_{\theta}\, V$ and $A^{}_{\sigma}\# {}^{Q}_{}\, V$ are all BiHom-Frobenius algebra structures on $E$.
Conversely,  any BiHom-Frobenius algebraic extending system $E$ of ${A}$  through $V$ is isomorphic to such two types.
Now from Theorem \ref{thm-41}, Theorem \ref{thm-42} we obtain the main result of in this section,
which solves the extending problem for BiHom-Frobenius algebra.

\begin{theorem}
Let $({A}, \cdot, \Delta_A)$ be a BiHom-Frobenius algebra, $E$ a BiHom-vector space containing ${A}$ as a subspace and $V$ be a complement of ${A}$ in $E$.
Denote by
$$\mathcal{HLB}(V,{A}):=\mathcal{IB}^{(1)}({A},V)\sqcup\mathcal{IB}^{(2)}({A},V)/\equiv.$$
Then the map
\begin{eqnarray*}
&&\Upsilon: \mathcal{HLB}(V,{A})\rightarrow BExtd(E,{A}),\\
&&\overline{\Omega^{(1)}({A},V)}\mapsto A^{P}_{}\# {}^{}_{\theta}\, V,\quad   \overline{\Omega^{(2)}({A},V)}\mapsto A^{}_{\sigma}\# {}^{Q}_{}\, V
\end{eqnarray*}
is bijective, where $\overline{\Omega^{(i)}({A}, V)}$ is the equivalence class of $\Omega^{(i)}({A}, V)$ under $\equiv$.
\end{theorem}

\section{Conclusions and problems}
In this paper,  we developed the theory of extending structures for BiHom-Frobenius algebras.
We find that  the concept of braided BiHom-Frobenius algebra plays a key role in the extending problem of BiHom-Frobenius algebra.
It is a natural question how to develop the theory of flag extending structures for BiHom-Frobenius algebras  as did in \cite{AM1}--\cite{AM6}.
Since this problem is more complicated than the ordinary algebra cases, the solutions to them are left to future investigations.


\vskip7pt
\footnotesize{
\noindent Tao Zhang\\
College of Mathematics and Information Science,\\
Henan Normal University, Xinxiang 453007, P. R. China;\\
 E-mail address: \texttt{{zhangtao@htu.edu.cn}}

\vskip7pt
\footnotesize{
\noindent Hui-jun Yao\\
College of Mathematics and Information Science,\\
Henan Normal University, Xinxiang 453007, PR China};\\
 E-mail address: \texttt{{yhjdyxa@126.com}}


\begin{thebibliography}{AM1}




\bibitem{AM1}
A. L. Agore, G. Militaru,  \emph{Extending structures I: the level of groups}, {Algebr. Represent. Theory} {17} (2014), 831--848.

\bibitem{AM2}
A. L. Agore, G. Militaru, \emph{Extending structures II: the quantum version},  J. Algebra {336} (2011), 321--341.


\bibitem{AM3}
A.L. Agore, G. Militaru, \emph{Extending structures for Lie algebras}, Monatsh. fur Mathematik 174(2014), 169--193.

\bibitem{AM4}
A. L. Agore, G. Militaru, \emph{Unified products for Leibniz algebras. Applications}, Linear Algebra Appl. {439} (2013), 2609--2633.


\bibitem{AM5}
A.L. Agore, G. Militaru, \emph{The global extension problem, crossed products and co-flag noncommutative Poisson algebras}, J. Algebra 426(2015), 1--31.

\bibitem{AM6}
A. L. Agore, G. Militaru, \emph{Extending structures, Galois groups and supersolvable associative algebras}, Monatsh. Math. {181} (2016), 1--33.



\bibitem{AS1}
N. Andruskiewitsch, H.-J. Schneider,
\emph{On the classification of finite-dimensional pointed Hopf algebras}, Ann. Math. 171(2010), 375--417.



\bibitem{Bai10} C. Bai,  \emph{Double constructions of Frobenius algebras, Connes cocycles and their duality}, J. Noncommut. Geom. 4 (2010), 475--530.



\bibitem{BD99}Y.~Bespalov, B.~Drabant,  \emph{Cross product bialgebras Part I},  J. Algebra, {219} (1999), 466--505.

\bibitem{BD01}Y.~Bespalov, B.~Drabant,  \emph{Cross product bialgebras Part II}, J. Algebra, {240} (2001), 445--504.



\bibitem{Dr86} V. G.~Drinfel'd,    \emph{Quantum groups}. In ``Proceedings International Congress of Mathematicians, August 3-11, 1986, Berkeley, CA" pp. 798--820, Amer. Math. Soc., Providence, RI, 1987.


\bibitem{GMMP}
G. Graziani, A. Makhlouf, C. Menini, F. Panaite, \emph{BiHom-associative algebras, BiHom-Lie algebras and BiHom-bialgebras}, SIGMA 11 (2015) 086.

\bibitem{HHS}
M. N. Hounkonnou, G. D. Houndedji, S. Silvestrov, \emph{Double constructions of biHom-Frobenius algebras}, arXiv:2008.06645.



\bibitem{Liu19}
L. Liu, A. Makhlouf, C. Menini, F. Panaite, \emph{$\{\sigma, \tau\}$-Rota-Baxter operators, infinitesimal Hom-bialgebras and the associative (Bi)Hom-Yang-Baxter equation}, Can. Math. Bull. 62 (2019), 355--372.

\bibitem{Liu20}
L. Liu,  A. Makhlouf, C. Menini, F. Panaite,  \emph{BiHom-Novikov algebras and infinitesimal BiHom-bialgebras}, J. Algebra 560(2020), 1146--1172.

\bibitem{Ma90a}
S.~Majid,   \emph{Matched pairs of   groups associated to solutions of the Yang-Baxter equations}, Pacific J. Math.,  { 141} (1990), 311--332.

\bibitem{Ma95} S.~Majid,    \emph{Founditions of Quantum Groups}, Cambridge: Cambridge University  Press, 1995.

\bibitem{Ma00} S.~Majid,   \emph{Braided Lie bialgebras}, Pacific J. Math.,  { 192} (2000), 329--356.



\bibitem{Mas00} A.~Masuoka,   \emph{Extensions of Hopf algebras and  Lie bialgebras}, Trans. Amer. Math. Soc.,  { 352}(2000), 3837--3879.

\bibitem{Ra85} D.~E.~Radford,  \emph{The structure of Hopf algebras with a projection},  J. Algebra,  { 92} (1985), 322--347.


\bibitem{So96}  Y.~Sommerhauser,  \emph{Kac-Moody algebras}, Presented at the Ring Theory Conference, Miskolc, Hungary, Preprint, 1996.

\bibitem{S02}  Y.~Sommerhauser, \emph{Yetter-Drinfel'd Hopf algebras over groups of prime order}, Lect. Notes Math., 1789, Springer, Berlin, 2002.

\bibitem{yau} D.~Yau,   \emph{Infinitesimal Hom-bialgebras and Hom-Lie bialgebras}, (2010) arXiv.1001.5000v1.

\bibitem{Zh99} S.~C.~Zhang,  H.~X.~Chen,   \emph{The double bicrossproducts in braided tensor categories},  Comm. Algebra, {  29}(2001)(1), 31--66.

\bibitem{ZWJ98} W.~Z.~Zhao, S.~H.~Wang, Z.~M.~Jiao,  \emph{On the Hopf algebra structure over double crossproduct},   Comm. Algebra, { 26}(1998) 467--476.


\bibitem{Z1}T. Zhang, \emph{Double cross biproduct and bi-cycle bicrossproduct Lie bialgebras}, J. Gen. Lie Theory Appl. { 4} (2010), S090602.

\bibitem{Z2}T. Zhang, \emph{Extending structures for 3-Lie algebras},  Comm. Algebra  { 50} (2022), 1469--1497.


\bibitem{Z3}T. Zhang, \emph{Unified products for braided Lie bialgebras with applications}, J. Lie Theory { 32}(3) (2022), 671--696.

\bibitem{Z4}
T. Zhang, \emph{Extending structures for infinitesimal bialgebras}, 	arXiv:2112.11977v1.

\bibitem{ZY}T. Zhang and H. Yao,  \emph{Braided anti-flexible bialgebras},    arXiv:2208.02221,  to appear in J. Alg. Appl. (2024) 2450084 (38 pages)
DOI: 10.1142/S0219498824500841

\bibitem{ZCY}
J. Zhao, L. Chen, L. Yuan, \emph{Extending structures of  Lie conformal superalgebras}, Comm. Algebra 47(4)(2019), 1541--1555.



\end{thebibliography}
\end{document}